\documentclass{amsart}
\usepackage{amsfonts}
\usepackage{latexsym}
\usepackage{amssymb}
\usepackage{amsmath}
\usepackage{todonotes}
\usepackage{enumerate}

\newcommand{\R}{\mathbb R}

\newcommand{\conv}{\mathrm{conv}}

\newtheorem{thm}{Theorem}[section]
\newtheorem{lemma}[thm]{Lemma}

\newtheorem{cor}[thm]{Corollary}

\newtheorem{df}{Definition}[section]
\theoremstyle{remark}
\newtheorem*{rmk}{Remark}

\begin{document}


\title{Rogers-Shephard inequality for log-concave functions}

\author[David Alonso, Bernardo Gonz\'alez, C. Hugo Jim\'enez, Rafael Villa]{David Alonso-Guti\'{e}rrez, Bernardo Gonz\'alez Merino, C. Hugo Jim\'enez, Rafael Villa}

\email{alonsod@unizar.es}
\email{bg.merino@tum.de}
\email{hugojimenez@mat.puc-rio.br}
\email{villa@us.es}
\address{Universidad de Zaragoza}
\address{Technische Universit\"at M\"unchen}
\address{Pontif\'icia Universidade Cat\'olica do Rio de Janeiro}
\address{Universidad de Sevilla}



\date{\today}
\begin{abstract}
In this paper we prove different functional inequalities extending the classical Rogers-Shephard inequalities for convex bodies. The original inequalities provide an optimal relation between the volume of a convex body and the volume of several symmetrizations of the body, such as, its difference body. We characterize the equality cases in all these inequalities. Our method is based on the extension of the notion of a convolution body of two convex sets to
any pair of log-concave functions and the study of some geometrical properties of these new sets.
\end{abstract}
\maketitle

\section{Introduction}

A measure $\mu$ on $\R^n$ is log-concave if for any measurable sets $A,B\subset\R^n$ and $0<\lambda<1$,
$$
\mu(\lambda A+(1-\lambda)B)
\ge
\mu(A)^\lambda \mu(B)^{1-\lambda}
$$
whenever $A,B\subset\R^n$ and $\lambda A+(1-\lambda)B$ are measurable, where $A+B =\{a+b: a \in A, b\in B\}$ is the Minkowski sum.

Log-concave measures naturally appear in Convex Geometry, since the Brunn–
Minkowski inequality establishes the log-concavity of the Lebesgue measure restricted to convex sets, and of the marginal sections of convex sets.

A function $f :\R^n\to [0,+\infty)$ is
log-concave if $f(x)=e^{-u(x)}$ for some convex function $u:\R^n\to(-\infty,\infty]$. As was shown in \cite{Bo}, a measure $\mu$ on $\R^n$ with full-dimensional support is log-concave if and
only if it has a log-concave density with respect to the Lebesgue measure.

The class of log-concave functions has proven to be of great importance in several areas or mathematics. From a functional point of view it has been shown they resemble Gaussian functions in many different ways. Many functional inequalities satisfied by Gaussian functions, like Poincare and Log-Sobolev inequalities,  also hold in a more general subclass of log-concave functions \cite{BBCG,BE,B}.
They also appear in areas as Information Theory,
in the study of
some  important parameters, such as the classical entropy \cite{BM1}.
There are many examples in the literature of functional inequalities with a geometric counterpart; Prekopa-Leindler/Brunn-Minkowski \cite{P} and Sobolev/Petty projection \cite{Z} inequalities are two of the main examples. This has generated an increasing interest to extend several important parameters of convex bodies to functional parameters \cite{AKM,AKSW,B,FM,KM,M} in the class of log-concave functions.

The class of  log-concave functions is often regarded as the natural extension of convex bodies,
 taking into account that the characteristic function of a convex body is a log-concave function and that this is the smallest closed under limits class of functions that contains the densities of $n$-dimensional marginals of uniform probabilities on convex bodies of higher dimension (we refer to the next section for precise definitions).

In this work we extend Rogers-Shephard inequality \cite{RS,RS2} to the class of log-concave functions (Theorems \ref{RSfunctions} and \ref{thm:RSsurfacefunctions}), characterizing the equality cases as well. We provide  other functional versions of inequalities around Rogers-Shephard's with their respective characterization of equality cases. While the Brunn-Minkowski inequality is commonly seen as the backbone of modern Convex Geometry, Rogers-Shephard inequality can be considered as a reverse form of Brunn-Minkowski inequality that not only describes a relation between Minkowski addition and  volume, but also deals with yet another fundamental property in convexity: symmetry. The far reaching influence of this inequality becomes evident as it can be found as an ingredient not only in many important works in classical and asymptotic convex geometry \cite{Be,Ku,Kl,Pa} but also in many others with a more analytical flavor \cite{R,ASY,SWZ,MP} and its extension to a functional setting as well as for the entropy of convex measures has already been considered for instance in \cite{C} or \cite{BM2}.

The paper is organised as follows: In Section \ref{notationprevresults} we provide the notation
used in the rest of the paper and some previous results and state the precise results we are going to prove. In Section \ref{ProofInequalities} we introduce the $(\theta,t)$-convolution bodies of log-concave functions and prove the extension of Rogers-Shephard inequalities (\ref{eq:RSfunctions}) and (\ref{eq:f=-g}). In Section \ref{ProofInequalities2} we introduce the more general concept of $k$-th $(\theta,t)$-convolution bodies and prove a Rogers-Shephard type inequality (\ref{eq:RSsurfacefunctions}) for surface area. When particularizing to $k=n$ we obtain the previously introduced $(\theta,t)$-convolution bodies.
In Section \ref{EqualityCases} we characterize the equality cases in these inequalities. Finally, in Section \ref{Colesanti}, we revisit another result around Rogers-Shephard inequality \cite{C} by giving an extended version for two different functions. Thus, we extend inequality (\ref{RSColesanti}) for any two log-concave functions and characterize the equality cases. Since this inequality will strengthen another well known Rogers-Shephard inequality (\ref{RSuniontwobodies}) we will make use of these new tools to characterize the equality cases in inequality (\ref{RSuniontwobodies}).

\section{Notation and previous results}\label{notationprevresults}

The notation used in this paper is quite standard in modern convex geometry and consistent with for example \cite{S} and \cite{G}. A convex body is a subset of $\R^n$ that is convex, compact and has non-empty interior. It is said to be centrally symmetric if for any $x\in K$ we have that also $-x\in K$. When studying geometric properties of a convex body $K$ it is usually very convenient to construct another convex body from $K$ which is centrally symmetric. There are many ways to construct such a symmetrization. One of them is the so called difference body of $K$, which is the Minkowski sum of $K$ and $-K$. Let us recall that the Minkowski sum of two convex bodies $K$ and $L$ is defined as
\begin{align*}
K+L
=&
\{x+y\in\R^n\,:\,x\in K,y\in L\}
\\
=&
\{x\in\R^n\,:\,K\cap(x-L)\neq\emptyset\}.
\end{align*}

Brunn-Minkowski inequality (see, for instance, \cite{AGM} for several proofs and characterization of the equality cases) gives the following lower bound for the volume of the Minkowski sum of any two convex bodies $K,L\subseteq\R^n$:
$$
|K+L|^\frac{1}{n}\geq |K|^\frac{1}{n}+|L|^\frac{1}{n}.
$$

As a consequence of Brunn-Minkowski inequality the following relation between the volume of the difference body $K-K$ and the volume of $K$ is always true
$$
|K-K|\geq 2^n|K|,
$$
with equality if and only if $K$ is symmetric. In \cite{RS} Rogers and Shephard proved a reverse inequality. Namely,
Rogers-Shephard inequality states that for any  convex body $K\subseteq\R^n$, we have
\begin{equation}\label{RSonebody}
|K-K|\leq{2n\choose n}|K|,
\end{equation}
with equality if and only if $K$ is a simplex. This inequality was extended to any pair of convex bodies $K,L\subseteq\R^n$ in \cite{RS2}, showing that
\begin{equation}\label{RStwobodies}
\max_{x_0\in\R^n}|K\cap(x_0-L)||K+L|\leq{2n\choose n}|K||L|,
\end{equation}
with equality if and only if $K=-L$ is a simplex  (see \cite{AJV} for the characterization of equality).

In the same paper \cite{RS2}
the authors also considered different types of symmetrization of a convex body $K$ and proved volume inequalities for them. In particular it was shown that for any convex body $K\subseteq\R^n$ containing 0, the volume of the convex hull of $K$ and $-K$ verifies
\begin{equation}\label{RSunion}
|\conv\{K, -K\}|\leq 2^n|K|
\end{equation}
with equality if and only if $K$ is a simplex and $0$ is one of its vertices. In the same paper the authors remarked that, with a similar proof, the latter inequality can be extended to the following inequality for any two bodies $K$ and $L$ containing the origin
\begin{equation}\label{RSuniontwobodies}
|K\cap L||\conv\{K, -L\}|\leq 2^n|K||L|
\end{equation}
and they suggested that it is likely that equality is attained if and only if $K=L$ is a simplex and $0$ is one of its vertices.

Very recently, in \cite{AJV}, the volume of the $\theta$-convolution bodies $K+_\theta L$ was studied, where
\begin{equation}\label{ConvolutionBodies}
K+_\theta L=\{x\in K+L\,:\,|K\cap(x-L)|\geq \theta\max_{z\in\R^n}|K\cap(z-L)|\}.
\end{equation}
As a consequence of the volume inequalities obtained for convolution bodies, inequality (\ref{RStwobodies}) was recovered and the equality cases were characterized.

In \cite{AGJ}, similar inclusion relations and volume inequalities were obtained for the $h,\theta$-convolution bodies of $K$ and $L$, defined as
$$
K+_{h,\theta}L=\{x\in K+L\,:\,h(K\cap(x-L))\geq \theta\max_{z\in\R^n}h(K\cap(z-L))\},
$$
where $h$ is a function satisfying some properties. As a particular case we have the $k$-th $\theta$-convolution bodies of two convex bodies, defined as
$$
K+_{k,\theta} L=\{x\in K+L\,:\,W_{n-k}(K\cap(x-L))\geq \theta\max_{z\in\R^n}W_{n-k}(K\cap(z-L))\},
$$
where $W_{n-k}$ denotes the $(n-k)$-th querma\ss integral of a convex body which, by Kubota's formula (cf. \cite[p. 295]{S}),
can be expressed as an average of the volumes of the $k$-dimensional projections of $K$
$$
W_{n-k}(K)=\frac{|B_2^n|}{|B_2^k|}\int_{G_{n,k}}|P_E(K)|d\mu(E).
$$
($G_{n,k}$ denotes the set of $k$-dimensional linear subspaces, $d\mu$ is the Haar probability measure on $G_{n,k}$ and $P_E(K)$ is the projection of $K$ on a subspace $E$). As a consequence of these volume inequalities the following Rogers-Shephard type inequality for any two convex bodies $K,L\subseteq\R^n$, which involves the surface area of $K$ and $L$, is obtained
\begin{equation}\label{RSsurface}
|K+L|\leq{2n\choose n}\frac{|K||\partial L|+|L||\partial K|}{2\max_{x_0\in\R^n}|\partial(K\cap(x_0-L))|},
\end{equation}
where  $|\partial K|$ is the surface area of $K$. Notice that when $L=-K$ we recover inequality (\ref{RSonebody}). Let us recall that, up to a constant which depends only on the dimension $n$, the surface area of a convex body $K$ equals the querma\ss integral $W_1(K)$.

Inequality (\ref{RSunion}) was extended to the context of log-concave functions  in \cite{C}, where the author proved that for any log-concave function $f$, if its difference function is defined by
$$
\Delta f(z)=\sup
\left\{
\sqrt{f(x)f(-y)}\,:\,x,y\in\R^n\,:2z=x+y
\right\},
$$
then
\begin{equation}\label{RSColesanti}
\int_{\R^n}\Delta f(x)dx\leq 2^n\int_{\R^n}f(x)dx.
\end{equation}
Taking $f(x)=e^{-h_{K^\circ}(x)}$, with $h_{K^\circ}(x)=\max_{y\in K^\circ}\langle x,y\rangle$ the support function of the polar set of a convex body $K$ containing the origin,  inequality (\ref{RSunion}) is recovered. This inequality was also extended in \cite{AEFO}.

In this paper we extend inequalities (\ref{RSonebody}), (\ref{RStwobodies}) and (\ref{RSsurface}) to log-concave functions. Before we state our results we need to introduce some more notation.

Given $f,g$ two log-concave functions, their convolution  defined by
$$
f*g(x)=\int_{\R^n}f(z)g(x-z)dz
$$
 is also a log-concave function in $\R^n$. If $f(x)=\chi_K(x)$ and $g(x)=\chi_L(x)$ are the characteristic functions of two convex bodies, then $f*g(x)=|K\cap(x-L)|$.

The Asplund product of two log-concave functions is defined by
$$
f\star g(x)=\sup_{z\in\R^n}f(z)g(x-z).
$$
If $f(x)=\chi_K(x)$ and $g(x)=\chi_L(x)$, then $f\star g(x)=\chi_{K+L}(x)$. This operation is the natural extension of the Minkowski sum of convex bodies,
as it has been shown when extending geometric inequalities to the context of general log-concave functions (see for instance \cite{AKM}).

\begin{rmk}\label{maximumAsplund}
Notice that if both $f$ and $g$ are integrable and continuous when restricted to their supports then this supremum is a maximum since, in such case, if $f\star g(x)=0$, then for any $z\in\R^n$ we have that $f(z)g(x-z)=0$ and if $f\star g(x)>0$, then there exists a $t>0$ such that the set
$$
\mathcal{A}_{t}(x):=\{z\in\textrm{supp } f\cap (x-\textrm{supp }g)\,:\,f(z)g(x-z)\geq t\Vert f\Vert_\infty\Vert g\Vert_\infty\}
$$
is not empty. Since $f$ and $g$ are integrable log-concave functions this set is convex and bounded. Thus, its closure is a compact convex set. Since both $f$ and $g$ are continuous when restricted to their supports and
$$
f\star g(x)=\sup_{z\in\R^n}f(z)g(x-z)=\sup_{z\in\textrm{cl}(\mathcal{A}_{t}(x))}f(z)g(x-z)
$$
the function $f(z)g(x-z)$ is continuous on the compact set $\mathcal{A}_{t}(x)$ and the maximum is attained.
\end{rmk}

With this notation, we  prove the following extension of inequality (\ref{RStwobodies}).
\begin{thm}\label{RSfunctions}
Let $f,g:\R^n\to\R$ be two integrable log-concave functions with full-dimensional support such that $f$ and $g$ are continuous when restricted to their supports. Then
\begin{equation}\label{eq:RSfunctions}
\Vert f*g\Vert_\infty\int_{\R^n}f\star g(x)dx\leq{2n\choose n}\Vert f\Vert_\infty\Vert g\Vert_\infty\int_{\R^n}f(x)dx\int_{\R^n}g(x)dx.
\end{equation}
Furthermore, this inequality becomes an equality if and only if $\frac{f(x)}{\Vert f\Vert_\infty}=\frac{g(-x)}{\Vert g\Vert_\infty}$ is the characteristic function of an $n$-dimensional simplex.
\end{thm}

In case we consider $g(x)=f(-x)$ the latter inequality can be improved to the following extension of inequality (\ref{RSonebody}):

\begin{thm}\label{thm:g=f-}
Let $f$ be a log-concave function with full-dimensional support and continuous when restricted to it and let $\bar f(x)=f(-x)$. Then
\begin{equation}\label{eq:f=-g}
\int_{\R^n}f\star \bar f(x)dx\leq{2n\choose n}\Vert
f\Vert_\infty\int_{\R^n} f(x)dx.
\end{equation}
Furthermore, this inequality becomes an equality if and only if $\frac{f(x)}{\Vert f\Vert_\infty}$ is the characteristic function of an $n$-dimensional simplex.
\end{thm}

Let us mention that the previous inequality was first obtained by Colesanti \cite[Theorem 4.3]{C} where the author proves it in the quasi-concave case without characterizing the equality case.

The fact that inequality (\ref{eq:f=-g}) is an improvement of inequality (\ref{eq:RSfunctions}) follows from Young's inequality $\Vert f*\bar f\Vert_\infty\leq\Vert f\Vert_1\Vert f\Vert_\infty$.

The notion of querma\ss integrals has also been extended from convex bodies to the setting of log-concave functions. In \cite{CF}, \cite{KM} and \cite{Ro}, the case of the perimeter and the mean width is considered while, in \cite{BCF}, a different definition is given for all the querma\ss integrals. We will work with the definition in the latter paper, where, in particular, the querma\ss integral $W_1$ (surface area) of a log-concave function is defined by
$$
W_1(f):=\int_0^\infty W_1(\{x\in\R^n\,:\, f(x)\geq t\})dt.
$$
By Crofton's formula (cf. \cite[p. 235]{S}), this equals
$$
W_1(f)=c_n\int_{\mathbb A_{n,1}}\max_{z\in E}f(z)d\mu_{n,1}(E),
$$
where $c_n=\frac{|B_2^n|}{|B_2^{n-1}|}$ is a constant depending only on $n$ and ${\mathbb A_{n,1}}$ is the set of affine 1-dimensional subspaces of $\R^n$ and $\mu_{n,1}$ is the Haar probability measure on it.

\begin{thm}\label{thm:RSsurfacefunctions}
Let $f,g:\R^n\to\R$ be two integrable log-concave functions with full-dimensional support and continuous when restricted to their supports. Then
\begin{equation}\label{eq:RSsurfacefunctions}
\int_{\R^n}f\star g(x)dx
\leq
{2n \choose n}\Vert f\Vert_\infty\Vert g\Vert_\infty
\frac
{W_1(g)\int_{\R^n}f(x)dx+W_1(f)\int_{\R^n}g(x)dx}
{2\max_{x_0\in\R^n}W_1(f(\cdot)g(x_0-\cdot))}.
\end{equation}
Furthermore, when $n\geq 3$ this inequality becomes an equality if and only if $\frac{f(x)}{\Vert f\Vert_\infty}=\frac{g(-x)}{\Vert g\Vert_\infty}$ is the characteristic function of an $n$-dimensional simplex.
\end{thm}
Finally, we will prove the following extension of (\ref{RSColesanti}). Before stating it let us start with the following definition:

\begin{df}
Let $f, g:\R^n\rightarrow\R$ be two integrable log-concave functions. Define
\[
f\oplus g(z):=\sup_{2z=x+y}\sqrt{f(x)g(y)}=\sqrt{f\star g(2z)}.
\]
\end{df}
Following the proof given in \cite{C} of inequality (\ref{RSColesanti}), we can show the following result. The inequality in the result was also obtained in \cite{AEFO} in the more general case where $z=\lambda x+ (1-\lambda)y$ and not just $\lambda=\frac{1}{2}$.
However, equality cases need a more detailed argument.

\begin{thm}\label{th_Colfg}
Let $f, g:\R^n\rightarrow\R$ be two integrable log-concave functions with full-dimensional supports and continuous when restricted to them. Then
\begin{equation}\label{eq:Colfg}
\int_{\R^n}\sqrt{f(x)\bar g(x)}dx\int_{\R^n}f\oplus g(x)dx\leq
2^n\int_{\R^n}f(x)dx\int_{\R^n}g(x)dx.
\end{equation}
Equality holds if and only if the following two conditions are satisfied:
\begin{enumerate}[(i)]
\item $\textrm{supp }f=\textrm{supp }\bar g$ is a translation of a cone $C$ with vertex at 0 with simplicial section, and
\item $f(x)=c_1 e^{-\langle a,x\rangle}$ on $\textrm{supp }f$ and $g(x)=c_2 e^{-\langle b,x\rangle}$ on $\textrm{supp }g$ for some $c_1,c_2>0$ and some $a,b\in\R^n$ such that $\langle a,x\rangle\geq0\geq\langle b,x\rangle$ for every $x\in C$.
\end{enumerate}
\end{thm}

\section{$(\theta,t)$-convolution bodies of log-concave functions and Rogers-Shephard inequalities}\label{ProofInequalities}

In this section we prove the aforementioned extensions of Rogers-Shephard inequality to log-concave functions. In order to prove them we need to introduce some more notation. Given $f,g$ two integrable log-concave functions with full-dimensional support, $x\in\textrm{supp } f+\textrm{supp } g$ and $t\in(0,1]$, let us recall that we denote
$$
\mathcal{A}_{t}(x)=\mathcal{A}_{t}(f,g)(x):=\{z\in\textrm{supp } f\cap (x-\textrm{supp }g)\,:\,f(z)g(x-z)\geq t\Vert f\Vert_\infty\Vert g\Vert_\infty\}.
$$
Since $f$ and $g$ are integrable log-concave functions, $\mathcal{A}_{t}(x)$ is a bounded convex set.

\begin{df}
Let $f,g:\R^n\to\R$ be two  integrable log-concave functions with full-dimensional support, $t\in(0,1]$, $\theta\in[0,1]$. We define the $(\theta,t)$-convolution set of $f$ and $g$ as the set
$$
\mathcal{C}_{\theta, t}
=
\mathcal{C}_{\theta, t}(f,g)
:=
\left\{x\in\textrm{supp } f+\textrm{supp } g\,:\,\mathcal{A}_{t}(x)\neq\emptyset\,,\,|\mathcal{A}_{t}(x)|\geq\theta M_{t}\right\}
$$
where
$$
M_{t}
=
M_{t}(f,g)
:=
\max_{x_0\in\textrm{supp } f+\textrm{supp } g}|\mathcal{A}_{t}(x_0)|.
$$
\end{df}
\begin{rmk}
When $f(x)=\chi_K(x)$ and $g(x)=\chi_L(x)$ are the characteristic functions of two convex bodies $K$ and $L$, the sets $\mathcal{A}_{t}(x)=K\cap(x-L)$ for any $t\in(0,1]$ and we recover the definition of the $\theta$-convolution bodies $K+_\theta L$ in \cite{AJV}.
\end{rmk}

It is obvious from the definition that for any fixed $t$, the sets $\mathcal{C}_{\theta,t}$ decrease on $\theta$. The following lemma implies the convexity of these sets and gives a reverse (increasing on $\theta$) inclusion relation when normalized by the right factor, as shows Corollary \ref{inclusionrelation}.

\begin{lemma}
Let $t\in(0,1]$, $f,g:\R^n\to\R$ be two integrable log-concave functions with full-dimensional support such that $M_{t}= |\mathcal{A}_{t}(0)|$, $\theta_1,\theta_2,\lambda_1,\lambda_2\in[0,1]$ with $\lambda_1+\lambda_2\leq 1$. Then
$$
\lambda_1\mathcal{C}_{\theta_1,t}+\lambda_2\mathcal{C}_{\theta_2,t}\subseteq \mathcal{C}_{\theta,t},
$$
with $1-\theta^\frac{1}{n}=\lambda_1(1-\theta_1^\frac{1}{n})+\lambda_2(1-\theta_2^\frac{1}{n})$.
\end{lemma}

\begin{proof}
Let $x_1\in \mathcal{C}_{\theta_1,t}$, $x_2\in \mathcal{C}_{\theta_2,t}$. For any  $z_0\in \mathcal{A}_{t}(0)$, $z_1\in \mathcal{A}_{t}(x_1)$, $z_2\in \mathcal{A}_{t}(x_2)$, the log-concavity of  $f$ and $g$ implies
\begin{align*}
&f((1-\lambda_1-\lambda_2)z_0+\lambda_1z_1+\lambda_2z_2)g(\lambda_1x_1+\lambda_2x_2-(1-\lambda_1-\lambda_2)z_0-\lambda_1z_1-\lambda_2z_2)\\
&\geq(f(z_0)g(-z_0))^{1-\lambda_1-\lambda_2}(f(z_1)g(x_1-z_1))^{\lambda_1}(f(z_2)g(x_2-z_2))^{\lambda_2}\geq t\Vert f\Vert_\infty\Vert g\Vert_\infty.
\end{align*}
Thus,
$$
\mathcal{A}_{t}(\lambda_1x_1+\lambda_2x_2)\supseteq(1-\lambda_1-\lambda_2)\mathcal{A}_{t}(0)+\lambda_1\mathcal{A}_{t}(x_1)+\lambda_2\mathcal{A}_{t}(x_2)
$$
and, by Brunn-Minkowski inequality
\begin{align*}
&|\mathcal{A}_{t}(\lambda_1x_1+\lambda_2x_2)|^{\frac{1}{n}}\geq\\
&(1-\lambda_1-\lambda_2)|\mathcal{A}_{t}(0)|^\frac{1}{n}+\lambda_1|\mathcal{A}_{t}(x_1)|^\frac{1}{n}+\lambda_2|\mathcal{A}_{t}(x_2)|^\frac{1}{n}\geq\\
&(1-\lambda_1-\lambda_2)M_{t}^\frac{1}{n}+\lambda_1\theta_1^\frac{1}{n}M_{t}^\frac{1}{n}+\lambda_2\theta_2^\frac{1}{n}M_{t}^\frac{1}{n}=\\
&(1-\lambda_1(1-\theta_1^\frac{1}{n})-\lambda_2(1-\theta_2^\frac{1}{n}))M_{t}^\frac{1}{n}.
\end{align*}
Consequently, $\lambda_1x_1+\lambda_2x_2\in \mathcal{C}_{\theta,t}$.
\end{proof}
In particular, taking $\theta_1=\theta_2$ and $\lambda_1+\lambda_2=1$ we obtain that these sets $\mathcal{C}_{\theta,t}$ are convex. Besides
\begin{cor}\label{inclusionrelation}
Let $t\in(0,1]$, $f,g:\R^n\to\R$ be two integrable log-concave functions  with full-dimensional support such that $M_{t}= |\mathcal{A}_{t}(0)|$, $0\leq\theta_0\leq\theta<1$. Then
$$
\frac{\mathcal{C}_{\theta_0,t}}{1-\theta_0^\frac{1}{n}}\subseteq\frac{\mathcal{C}_{\theta,t}}{1-\theta^\frac{1}{n}}.
$$
\end{cor}

\begin{proof}
Taking $\theta_1=\theta_2=\theta_0$ in the previous lemma, we have that for any $\lambda_1,\lambda_2\in[0,1]$ with $\lambda_1+\lambda_2\leq1$
$$
(\lambda_1+\lambda_2)\mathcal{C}_{\theta_0,t}=\lambda_1\mathcal{C}_{\theta_0,t}+\lambda_2\mathcal{C}_{\theta_0,t}\subseteq \mathcal{C}_{\theta,t},
$$
with $(\lambda_1+\lambda_2)(1-\theta_0^\frac{1}{n})=1-\theta^\frac{1}{n}$. Thus, for any $\theta_0\leq\theta\leq 1$, taking $\lambda_1+\lambda_2=\frac{1-\theta^\frac{1}{n}}{1-\theta_0^\frac{1}{n}}$ we obtain the result.
\end{proof}

In a similar way we can prove the following:
\begin{lemma}\label{ContinuityM_t}
Let $f,g:\R^n\to\R$ be two integrable log-concave functions with full-dimensional support. Then for any $t_1,t_2\in(0,1]$ and any $\lambda\in[0,1]$ we have
$$
M_{t_1^\lambda t_2^{1-\lambda}}^{\frac{1}{n}}\geq\lambda M_{t_1}^{\frac{1}{n}}+(1-\lambda)M_{t_2}^{\frac{1}{n}}.
$$
Consequently, $M_t$ is continuous on $(0,1)$.
\end{lemma}

\begin{proof}
Let $x_1,x_2$ be such that $M_{t_i}=|\mathcal{A}_{t_i}(x_i)|$ for $i=1,2$. Since $f$ and $g$ are log-concave we have that for any $z_1\in\mathcal{A}_{t_1}(x_1)$, $z_2\in\mathcal{A}_{t_2}(x_2)$ and any $\lambda\in,[0,1]$
\begin{eqnarray*}
&&f(\lambda z_1+(1-\lambda)z_2)g(\lambda x_1+(1-\lambda)x_2-(\lambda z_1+(1-\lambda)z_2))\cr
&\geq&(f(z_1)g(x_1-z_1))^\lambda(f(z_2)g(x_2-z_2))^{1-\lambda}\cr
&\geq&t_1^\lambda t_2^{1-\lambda}\Vert f\Vert_\infty\Vert g\Vert_\infty.
\end{eqnarray*}
Thus,
$$
\mathcal{A}_{t_1^\lambda t_2^{1-\lambda}}(\lambda x_1+(1-\lambda)x_2)\supseteq \lambda\mathcal{A}_{t_1}(x_1)+(1-\lambda)\mathcal{A}_{t_2}(x_2).
$$
By Brunn-Minkowsi inequality
$$
M_{t_1^\lambda t_2^{1-\lambda}}^\frac{1}{n}\geq|\mathcal{A}_{t_1^\lambda t_2^{1-\lambda}}(\lambda x_1+(1-\lambda)x_2)|^\frac{1}{n}\geq  \lambda M_{t_1}^\frac{1}{n}+(1-\lambda)M_{t_2}^\frac{1}{n}.
$$
Consequently, the function $f(s)=M_{e^{s}}^\frac{1}{n}$ is concave in $(-\infty,0]$ and then it is continuous on $(-\infty,0)$. Thus, $M_t=f^n(\log t)$ is continuous on $(0,1)$.
\end{proof}
Let us now prove inequality (\ref{eq:RSfunctions}).

\begin{proof}[Proof of Theorem \ref{RSfunctions} (inequality)]

We can assume, without loss of generality, that $\Vert f\Vert_\infty=\Vert g\Vert_\infty=1$. By definition of $\mathcal{C}_{\theta,t}$, we have that for any $t\in(0,1]$
\begin{align*}
\mathcal{C}_{0,t}&=\left\{x\in\textrm{supp} f+\textrm{supp} g\,:\,\mathcal{A}_{t}(x)\neq\emptyset\right\}\\
&=\left\{x\in\textrm{supp} f+\textrm{supp} g\,:\,f\star g(x)\geq t\right\}.
\end{align*}
For any $t\in(0,1]$, let $x_0(t)\in\R^n$ be such that $M_t=|\mathcal{A}_t(x_0)|$.
By Corollary \ref{inclusionrelation} with $\theta_0=0$ and $g$ replaced by $g(\cdot+x_0)$, for any $\theta\in [0,1]$
$$
(1-\theta^\frac{1}{n})(-x_0(t)+\mathcal{C}_{0,t})\subseteq -x_0(t)+\mathcal{C}_{\theta,t}.
$$
Taking volumes and integrating in $\theta\in[0,1]$ we obtain
$$
|\mathcal{C}_{0,t}|
\leq{2n\choose n}\int_0^1|\mathcal{C}_{\theta,t}|d\theta
={2n\choose n}\int_{\R^n}\frac{|\mathcal{A}_{t}(x)|}{M_{t}}dx.
$$
Consequently
$$
M_{t}|\mathcal{C}_{0,t}|\leq{2n\choose n}\int_{\R^n}|\mathcal{A}_{t}(x)|dx
$$
and, integrating in $t\in(0,1]$
$$
\int_0^1M_{t}|\mathcal{C}_{0,t}|dt
\leq{2n\choose n}\int_0^1\int_{\R^n}|\mathcal{A}_{t}(x)|dxdt.
$$
The integral on the right-hand side is
$$
\int_0^1\int_{\R^n}|\mathcal{A}_{t}(x)|dxdt
=\int_{\R^n}\int_{\R^n}f(z)g(x-z)dzdx
=\int_{\R^n}f(x)dx\int_{\R^n}g(x)dx.
$$
On the other hand, the integral on the left-hand side is
\begin{align*}
\int_0^1M_{t}|\mathcal{C}_{0,t}|dt&=\int_{\R^n}\int_0^{f\star g(x)}\max_{x_0\in\R^n}|\mathcal{A}_{t}(x_0)|dtdx\\
&\geq\max_{x_0\in\R^n}\int_{\R^n}\int_0^{f\star g(x)}|\mathcal{A}_{t}(x_0)|dtdx\\
&=\max_{x_0\in\R^n}\int_{\R^n}\int_{\R^n}\min\left\{f\star g(x),f(z)g(x_0-z)\right\}dzdx.\\
\end{align*}
Since $\Vert f\Vert_\infty=\Vert g\Vert_\infty=1$, both quantities in the minimum are smaller than or equal to 1, the minimum is bounded from below by the product and so, this quantity is bounded from below by
$$
\max_{x_0\in\R^n}f*g(x_0)\int_{\R^n}f\star g(x)dx.
$$
Thus
$$
\Vert f*g\Vert_\infty\int_{\R^n}f\star g(x)dx\leq{2n\choose n}\Vert f\Vert_\infty\Vert g\Vert_\infty\int_{\R^n}f(x)dx\int_{\R^n}g(x)dx.
$$
\end{proof}
Let us now consider the case in which $g(x)=f(-x)$. We will denote this function $\bar f$ and $\mathcal{A}_{t}(f)(x):=\mathcal{A}_{t}(f,\bar f)(x)$.
Notice that for any $t\in(0,1)$
\begin{align*}
\mathcal{A}_{t}(f)(0)
&=\left\{z\in\R^n\,:\,f(z)^2\geq t\Vert f\Vert_\infty^2\right\}
\\
&=\left\{z\in\R^n\,:\,f(z)\geq \sqrt t\Vert f\Vert_\infty\right\}.
\end{align*}

Analogously, let us denote $M_t(f):=M_t(f,\bar f)$ and $\mathcal{C}_{\theta,t}(f):=\mathcal{C}_{\theta,t}(f,\bar f)$.

The following lemma shows that  the maximum value of $|\mathcal{A}_{t}(f)(x)|$ is attained at $x=0$.
\begin{lemma}
Let $f:\R^n\to\R$ be an integrable log-concave function with full-dimensional support, then for any $t\in(0,1)$,
$$
\mathcal{A}_{t}(f)(x)\subseteq \frac{1}{2}x+\mathcal{A}_{t}(f)(0).
$$
Consequently, $M_{t}(f)=|\mathcal{A}_{t}(f)(0)|$.
\end{lemma}
\begin{proof}
Since $f$ is log-concave, for any $x\in\textrm{supp }f-\textrm{supp }f$
\begin{align*}
\mathcal{A}_{t}(f)(x)
&=\left\{z\in\R^n\,:\,\sqrt{\frac{f(z)}{\Vert f\Vert_\infty}\frac{f(z-x)}{\Vert f\Vert_\infty}}\geq\sqrt t\right\}
\\
&\subseteq\left\{z\in\R^n\,:\,\frac{f(z-\frac{1}{2}x)}{\Vert f\Vert_\infty}\geq\sqrt t\right\}\\
&=\frac{1}{2}x+\mathcal{A}_{t}(f)(0).
\end{align*}
\end{proof}
The following lemma shows a relation between the $(\theta,t)$-convolution bodies of $f$ and $\bar f$ and the $\theta$-convolution bodies of $\mathcal{A}_{t}(f)(0)$ and $-\mathcal{A}_{t}(f)(0)$.

\begin{lemma}\label{Inclusions_onefunction}
Let $f:\R^n\to\R$ be an integrable log-concave function with full-dimensional support. Then, for any $\theta\in[0,1]$
and any $t\in(0,1]$
$$
\mathcal{A}_{t}(f)(0)+_{\theta}(-\mathcal{A}_{t}(f)(0))
\subseteq
\mathcal{C}_{\theta,t}(f)
\subseteq
\mathcal{A}_{t^2}(f)(0)+_{\frac{M_{t}}{M_{t^2}}\theta}(-\mathcal{A}_{t^2}(f)(0)).
$$
\end{lemma}
\begin{proof}
Notice that
\begin{align*}
\mathcal{A}_{t}(f)(x)
&=
\left\{z\in\R^n\,:\,\frac{f(z)}{\Vert f\Vert_\infty}\frac{f(z-x)}{\Vert f\Vert_\infty}\geq t\right\}
\\
&\subseteq\left\{z\in\R^n\,:\,\min\left\{\frac{f(z)}{\Vert f\Vert_\infty},\frac{f(z-x)}{\Vert f\Vert_\infty}\right\}\geq t\right\}
\\
&=\mathcal{A}_{t^2}(f)(0)\cap(x+\mathcal{A}_{t^2}(f)(0)),\\
\end{align*}
which proves the right-hand side inequality. On the other hand, we trivially have
\begin{align*}
\mathcal{A}_{t}(f)(x)
&=\left\{z\in\R^n\,:\,\frac{f(z)}{\Vert f\Vert_\infty}\frac{f(z-x)}{\Vert f\Vert_\infty}\geq t\right\}
\\
&\supseteq\left\{z\in\R^n\,:\,\min\left\{\frac{f(z)}{\Vert f\Vert_\infty},\frac{f(z-x)}{\Vert f\Vert_\infty}\right\}\geq \sqrt t\right\}
\\
&=\mathcal{A}_{t}(f)(0)\cap(x+\mathcal{A}_{t}(f)(0)),
\end{align*}
which proves the left-hand side inequality.
\end{proof}

Using these relations we are now able to prove inequality (\ref{eq:f=-g}).

\begin{proof}[Proof of Theorem \ref{thm:g=f-} (inequality)]
From the right-hand side inequality in Lemma \ref{Inclusions_onefunction} we have that for any $t\in(0,1]$
$$\mathcal{C}_{0,t}(f)\subseteq \mathcal{A}_{t^2}(f)(0)-\mathcal{A}_{t^2}(f)(0).$$
Thus, taking volumes and using Rogers-Shephard inequality (\ref{RSonebody}) we have
$$
|\mathcal{C}_{0,t}(f)|\leq {2n \choose n}|\mathcal{A}_{t^2}(f)(0)|.
$$
Integrating $t\in (0,1]$ we obtain the result.
\end{proof}

\section{$k$-th $(\theta,t)$-convolution bodies of log-concave functions and Rogers-Shephard inequality}\label{ProofInequalities2}

In this section we will prove inequality (\ref{eq:RSsurfacefunctions}). To this end, we consider the following sets.
\begin{df}
Let $f,g:\R^n\to\R$ be two  integrable log-concave functions with full-dimensional support, $t\in(0,1]$, $\theta\in[0,1]$. We define the $k$-th $(\theta,t)$-convolution set of $f$ and $g$
$$
\mathcal{C}^k_{\theta, t}=
\mathcal{C}^k_{\theta, t}(f,g):=
\left\{x\in\textrm{supp } f+\textrm{supp } g\,:\,\mathcal{A}_{t}(x)\neq\emptyset\,,\,W_{n-k}(\mathcal{A}_{t}(x))\geq\theta M_{n-k,t}\right\}$$
where
$$
M_{n-k,t}=M_{n-k,t}(f,g):=\max_{x_0\in\textrm{supp } f+\textrm{supp } g}W_{n-k}(\mathcal{A}_{t}(x_0)).
$$
\end{df}
\begin{rmk}
When $f(x)=\chi_K(x)$ and $g(x)=\chi_L(x)$ are the characteristic functions of two convex bodies, the sets $\mathcal{A}_{t}(x)=K\cap(x-L)$ for any $t\in(0,1]$ and we recover the definition of the $k$-th $\theta$-convolution bodies $K+_{k,\theta} L$ in \cite{AGJ}.
\end{rmk}
When $k=n$ these are the convex bodies $\mathcal{C}_{\theta,t}$ introduced in the previous section and, with identical proofs, using the Brunn-Minkowski inequality for querma\ss integrals (see \cite{S}, Theorem 6.4.3) we have

\begin{lemma}
Let $t\in(0,1]$, $f,g:\R^n\to\R$ be two integrable log-concave integrable functions with full-dimensional support such that $M_{n-k,t}= W_{n-k}(\mathcal{A}_{t}(0))$, $\theta_1,\theta_2,\lambda_1,\lambda_2\in[0,1]$ with $\lambda_1+\lambda_2\leq 1$. Then
$$
\lambda_1\mathcal{C}^k_{\theta_1,t}+\lambda_2\mathcal{C}^k_{\theta_2,t}\subseteq \mathcal{C}^k_{\theta,t},
$$
with $1-\theta^\frac{1}{k}=\lambda_1(1-\theta_1^\frac{1}{k})+\lambda_2(1-\theta_2^\frac{1}{k})$.
\end{lemma}
Consequently,

\begin{cor}\label{inclusionrelationk}
Let $t\in(0,1]$, $f,g:\R^n\to\R$ be two integrable log-concave functions  with full-dimensional support such that $M_{n-k,t}= W_{n-k}(\mathcal{A}_{t}(0))$, $0\leq\theta_0\leq\theta<1$. Then
$$
\frac{\mathcal{C}^k_{\theta_0,t}}{1-\theta_0^\frac{1}{k}}\subseteq\frac{\mathcal{C}^k_{\theta,t}}{1-\theta^\frac{1}{k}}.
$$
\end{cor}

\begin{lemma}\label{ContinuityM_n-k,t}
Let $f,g:\R^n\to\R$ be two integrable log-concave functions with full-dimensional support. Then for any $t_1,t_2\in(0,1]$ and any $\lambda\in[0,1]$ we have
$$
M_{n-k,t_1^\lambda t_2^{1-\lambda}}^{\frac{1}{k}}\geq\lambda M_{n-k,t_1}^{\frac{1}{k}}+(1-\lambda)M_{n-k,t_2}^{\frac{1}{k}}.
$$
Consequently, $M_{n-k,t}$ is continuous on $(0,1)$.
\end{lemma}
Using the $(n-1)$-th $(\theta,t)$-convolution bodies of $f$ and $g$, we are now able to prove inequality (\ref{eq:RSsurfacefunctions}).

\begin{proof}[Proof of Theorem \ref{thm:RSsurfacefunctions} (inequality)]
We can assume, without loss of generality that $\Vert f\Vert_\infty=\Vert g\Vert_\infty=1$. Proceeding as in the proof of Theorem \ref{RSfunctions} we have that
\begin{align*}
\int_0^1M_{n-k,t}|\mathcal{C}^k_{0,t}|dt&\leq{n+k\choose n}\int_0^1\int_{\R^n}W_{n-k}(\mathcal{A}_{t}(x))dxdt.
\end{align*}
If $k={n-1}$
$$
\int_0^1M_{1,t}|\mathcal{C}^{n-1}_{0,t}|dt
=
\frac{1}{2}{2n\choose n}\int_0^1\int_{\R^n}W_{1}(\mathcal{A}_{t}(x))dxdt.
$$
Since $W_1(A)=c_n\int_{S^{n-1}}|P_{\theta^\perp}A|d\sigma(\theta)$, where $c_n$ is a constant depending only on $n$, the integral on the right hand side equals
\begin{align*}
&c_n\int_0^1\int_{\R^n}\int_{S^{n-1}}\int_{\theta^\perp}\chi_{\{\max_{s\in\R}f(z+s\theta)g(x-z-s\theta)\geq t\}}(z)dzd\sigma(\theta)dxdt
\\
=&c_n\int_{S^{n-1}}\int_{\theta^\perp}\int_{\R^n}\max_{s\in\R}f(z+s\theta)g(x-z-s\theta)dxdzd\sigma(\theta)
\\
=&c_n\int_{S^{n-1}}\int_{\theta^\perp}\int_{\theta^\perp}\int_\R\max_{s\in\R}f(z+s\theta)g(w-z+(r-s)\theta)drdwdzd\sigma(\theta)
\\
=&c_n\int_{S^{n-1}}\int_{\theta^\perp}\int_{\theta^\perp}\int_\R\max_{s\in\R}f(z+s\theta)g(w+(r-s)\theta)drdwdzd\sigma(\theta).
\end{align*}
Let $f_z(s)=f(z+s\theta)$ and $g_w(s)=g(w+s\theta)$ for fixed $\theta\in S^{n-1}$, $z,w\in\theta^\perp$. This quantity equals
$$
c_n\int_{S^{n-1}}\int_{\theta^\perp}\int_{\theta^\perp}\int_0^1\left|\left\{r\in\R\,:\,\max_{s\in\R}f_z(s)g_w(r-s)\geq t\right\}\right|dtdwdzd\sigma(\theta).
$$
Since the set in the integrand is contained in the set
\begin{align*}
&\left\{r\in\R\,:\,\max_{s\in\R}\min\left\{f_z(s)\Vert g_w\Vert_\infty,g_w(r-s)\Vert f_z\Vert_\infty\right\}\geq t\right\}
\\
=&\left\{r\in\R\,:\,\left\{s\,:f_z(s)\Vert g_w\Vert_\infty\geq t\right\}\cap\left(r-\left\{s\,:\,g_w(s)\Vert f_z\Vert_\infty\geq t\right\}\right)\neq\emptyset\right\}
\\
=&\left\{s\,:f_z(s)\Vert g_w\Vert_\infty\geq t\right\}+\left\{s\,:\,g_w(s)\Vert f_z\Vert_\infty\geq t\right\},
\end{align*}
and the 1-dimensional volume of the sum of segments is the sum of the volumes of the segments,
the previous integral is bounded from above by
\begin{align*}
&\le
c_n\int_{S^{n-1}}\int_{\theta^\perp}\int_{\theta^\perp}\int_0^{\Vert f_z\Vert_\infty\Vert g_w\Vert_\infty}\left|\left\{s\,:f_z(s)\Vert g_w\Vert_\infty\geq t\right\}\right|dtdwdzd\sigma(\theta)
\\
&+
c_n\int_{S^{n-1}}\int_{\theta^\perp}\int_{\theta^\perp}\int_0^{\Vert f_z\Vert_\infty\Vert g_w\Vert_\infty}\left|\left\{s\,:g_w(s)\Vert f_z\Vert_\infty\geq t\right\}\right|dtdwdzd\sigma(\theta)
\\
&=
c_n\int_{S^{n-1}}\int_{\theta^\perp}\int_{\theta^\perp}\Vert g_w\Vert_\infty\int_0^{\Vert f_z\Vert_\infty}\left|\left\{s\,:f_z(s)\geq t\right\}\right|dtdwdzd\sigma(\theta)
\\
&+
c_n\int_{S^{n-1}}\int_{\theta^\perp}\int_{\theta^\perp}\Vert f_z\Vert_\infty\int_0^{\Vert g_w\Vert_\infty}\left|\left\{s\,:g_w(s)\geq t\right\}\right|dtdwdzd\sigma(\theta)
\\
&=
c_n\int_{S^{n-1}}\int_{\theta^\perp}\int_{\theta^\perp}\Vert g_w\Vert_\infty\int_\R f_z(s) dsdwdzd\sigma(\theta)
\\
&+
c_n\int_{S^{n-1}}\int_{\theta^\perp}\int_{\theta^\perp}\Vert f_z\Vert_\infty\int_\R g_w(s)dtdwdzd\sigma(\theta)
\\
&=
c_n\int_{S^{n-1}}\int_{\theta^\perp}\Vert g_w\Vert_\infty\int_{\R^n}f(x) dxdwd\sigma(\theta)
\\
&+
c_n\int_{S^{n-1}}\int_{\theta^\perp}\Vert f_z\Vert_\infty\int_{\R^n}g(x)dxdzd\sigma(\theta)
\\
&=
c_n\int_{\mathbb A_{n,1}}\max_{z\in E}g(z)d\mu_{n,1}(E)\int_{\R^n}f(x) dx
\\
&+
c_n\int_{\mathbb A_{n,1}}\max_{z\in E}f(z)d\mu_{n,1}(E)\int_{\R^n}g(x)dx
\end{align*}
by Crofton's formula (cf. \cite[p. 235]{S}).

On the other hand, the integral in the left hand side is
\begin{align*}
\int_0^1M_{1,t}| & \mathcal{C}^{n-1}_{0,t}|dt=\int_{\R^n}\int_0^{f\star g(x)}\max_{x_0\in\R^n}W_1(\mathcal{A}_t(x_0))dtdx
\\
&\geq\max_{x_0\in\R^n}\int_{\R^n}\int_0^{f\star g(x)}W_1(\mathcal{A}_t(x_0))dtdx
\\
&=c_n\max_{x_0\in\R^n}\int_{\R^n}\int_0^{f\star g(x)}\int_{\mathbb A_{n,1}}\chi_{\{\max_{z\in E}f(z)g(x_0-z)\geq t\}}(E)d\mu_{n,1}(E)dtdx
\\
&=c_n\max_{x_0\in\R^n}\int_{\R^n}\int_{\mathbb A_{n,1}}\min\left\{f\star g(x),\max_{z\in E}f(z)g(x_0-z)\right\}d\mu_{n,1}(E)dx
\\
&\geq c_n \left(\int_{\R^n}f\star g(x)dx\right)\,\max_{x_0\in\R^n}\int_{\mathbb A_{n,1}}\max_{z\in E}f(z)g(x_0-z)d\mu_{n,1}(E).
\end{align*}
using Crofton's formula.
\end{proof}

\section{Equality cases in functional Rogers-Shephard inequalities}\label{EqualityCases}

In this section we will characterize the equality cases in Theorems \ref{RSfunctions}, \ref{thm:g=f-} and \ref{thm:RSsurfacefunctions}. Since the characterization of the equality cases in  Theorems \ref{RSfunctions} and \ref{thm:RSsurfacefunctions} follow the same lines we will write them together later and we will start with Theorem \ref{thm:g=f-}.

To characterize the equality in Theorem \ref{thm:g=f-} we need the following characterization of the characteristic function of a simplex:

\begin{lemma}
Let $f:\R^n\to\R$ be an integrable log-concave function with full-dimensional support such that $f$ is continuous when restricted to it and verifies $\Vert f\Vert_\infty=1$. Then $f(x)=\chi_{K}(x)$ with $K$ an $n$-dimensional simplex if and only if for every $t\in(0,1)$
\begin{enumerate}[(i)]
\item  $\mathcal{A}_{t}(f)(0)$ is an $n$-dimensional simplex, and
\item $\mathcal{A}_{1}(f)(0)$ contains a facet of $\mathcal{A}_{t}(f)(0)$.
\end{enumerate}
\end{lemma}

\begin{proof}
If $f$ is the characteristic function of a simplex then (i) and (ii) are trivially verified, since $\mathcal{A}_{t}(f)(x)=K\cap(x+K)$.
Assume that $f$ is a log-concave function with $\Vert f\Vert_\infty=1$, then $f(x)=e^{V(x)}$, with $V:\R^n\to[-\infty, 0]$ a concave function. Since
$$
\mathcal{A}_{t}(f)(0)=\{z\in\R^n\,:\,f(z)\geq\sqrt t\}=\{z\in\R^n\,:\,V(z)\geq\log(\sqrt t)\}
$$
we have that for every $s\in(-\infty,0]$, if $B_s$ denotes the set
$$
B_s=\{z\in\R^n\,:\,V(z)\geq s\},
$$
then
\begin{enumerate}[(i)]
\item $B_s$ is an $n$-dimensional simplex
\item $B_0$ contains a facet of $B_s$.
\end{enumerate}

Let $t_0<1$ and $s_0=\log\sqrt t_0<0$. Since $V$ is concave, the set
$$
G=\{(x,y)\in\R^{n+1}\,:\, V(x)\geq y\}
$$
is convex. Thus, $G$ contains $C=\textrm{conv}\{B_0\times\{0\},B_{s_0}\times\{s_0\}\}$. Consequently, for every $s\in[s_0,0)$ we have that
$$
\{z\in\R^n\,:\,(z,s)\in C\}\subseteq B_s\subseteq B_{s_0}.
$$

Since $B_0$ contains a facet $F$ of $B_{s_0}$ we have that for every $s\in [s_0,0)$, $F\times\{s\}$ is contained in $C$ and consequently $F$ is a facet of $B_s$ for every $s\in[s_0,0)$.

Let $P_{s_0}$ be the vertex of $B_{s_0}$ that is not contained in $F$ and let $P_1,\dots,P_n$ be the vertices of $B_{s_0}$ that are contained in $F$. For any choice of $i_1,\dots, i_{n-1}\in\{1,\dots, n\}$ we have that
$$
\textrm{conv}\{P_{i_1}\times\{0\},\dots,P_{i_{n-1}}\times\{0\}, P_{s_0}\times\{s_0\}\}\subseteq C
$$
and consequently, if $s\in (s_0,0)$ we can write $s=(1-(n-1)\lambda)s_0$ for some $\lambda$, and the point
$$
((1-(n-1)\lambda)P_{s_0}+\lambda P_{i_1}+\dots+\lambda P_{i_{n-1}})\times\{s\}
$$
belongs to $C$ and consequently to $B_s$. Since the point $(1-(n-1)\lambda)P_{s_0}+\lambda P_{i_1}+\dots+\lambda P_{i_{n-1}}$ belongs to the relative interior of the facet of $B_{s_0}$ spanned by $P_{s_0},P_{i_1},\dots,P_{i_{n-1}}$ we have that the facet $F$ spanned by $P_{1},\dots,P_n$ must be a facet of the simplex $B_s$, the vertex $P_{s}$ not contained in $F$ must belong to $B_{s_0}$ and, at the same time $B_s$ must contain some points lying in the relative interior of the other facets of $B_{s_0}$. Thus, $P_s$ must be equal to $P_{s_0}$ and for every $s\in(-\infty,0)$ the simplex $B_s$ must be the same.

Thus, for every $t\in(0,1)$ the simplex $\mathcal{A}_{t}(f)(0)$ is the same simplex and $f$ is the characteristic function of a simplex.
\end{proof}

Now we are able to prove the characterization of the equality cases in (\ref{eq:f=-g}).

\begin{proof}[Proof of Theorem \ref{thm:g=f-} (equality)]
We can assume, without loss of generality, that $\Vert f\Vert_\infty=1$. It is clear that if $f$ is a characteristic function of a simplex then the inequality becomes Rogers-Shephard inequality (\ref{RSonebody}) for a simplex, which is an equality.

Let us show that equality in (\ref{eq:f=-g}) implies that $f$ is the characteristic function of a simplex. We recall that equality holds in (\ref{eq:f=-g}) if and only if equality holds
in
$$
\int_0^1|\mathcal{C}_{0,t}(f)|dt
\stackrel{(a)}{\le}
\int_0^1|\mathcal{A}_{t^2}(f)(0)-\mathcal{A}_{t^2}(f)(0)|
\stackrel{(b)}{\le}
{2n\choose n}\int_0^1|\mathcal{A}_{t^2}(f)(0)|dt.
$$
Since $f$ is continuous when restricted to its support, $\mathcal{A}_{t}(f)(0)$ is a convex body for every $t\in(0,1)$ and the functions $|\mathcal{C}_{0,t}(f)|$, $|\mathcal{A}_{t^2}(f)(0)+_0(-\mathcal{A}_{t^2}(f)(0))|$ and $|\mathcal{A}_{t^2}(f)(0)|$ are continuous on $t\in (0,1]$ and so, equality in (a) and (b) implies that for every $t\in(0,1]$ there is equality in
$$
|\mathcal{C}_{0,t}(f)|\le|\mathcal{A}_{t^2}(f)(0)-\mathcal{A}_{t^2}(f)(0)|
\le{2n\choose n}|\mathcal{A}_{t^2}(0)|.
$$

First of all notice that, from the characterization of the equality cases in Rogers-Shephard inequality (\ref{RSonebody}) we have equality in the right hand side inequality if and only if for every $t\in(0,1)$ $\mathcal{A}_{t^2}(f)(0)$ is a simplex. We will see that $\mathcal{A}_{1}(f)(0)$ contains a facet of $\mathcal{A}_{t^2}(f)(0)$ for every $t\in(0,1)$.

Let us assume that there exists $t\in(0,1)$, and $P, Q$ vertices
of the simplex $\mathcal{A}_{t^2}(f)(0)$ such that $P, Q\notin\mathcal{A}_{1}(f)(0)$. We are going to see that in such case the set $\mathcal{C}_{0,t}(f)$ is strictly contained in $\mathcal{A}_{t^2}(f)(0)-\mathcal{A}_{t^2}(f)(0)$ and, since $f$ is continuous, these sets do not have the same volume.

Since $P$ and $Q$ are vertices of $\mathcal{A}_{t^2}(f)(0)$ we have that $f(P)=f(Q)=t$.
Let us take $\varepsilon>0$ such that $(t+\varepsilon)^2\leq t-\varepsilon$.
$\mathcal{A}_{1}(f)(0)$ does not contain $P$. Thus there exists a hyperplane $H_P$
separating them, parallel to the opposite facet to $P$ of $\mathcal{A}_{t^2}(f)(0)$. Notice that $H_P$ intersects
$\mathcal{A}_{t^2}(f)(0)$ and we can take $H_P$ such that $t\leq f(z)\leq t+\varepsilon$ for every $z\in \mathcal{A}_{t^2}(f)(0)$ in the same side of $H_P$ as $P$. Analogously, we take a hyperplane $H_Q$ separating $Q$ from $\mathcal{A}_{1}(f)(0)$ parallel to the opposite facet of $Q$ of $\mathcal{A}_{t^2}(f)(0)$ and such that $t\leq f(z)\leq t+\varepsilon$ for every $z\in \mathcal{A}_{t^2}(f)(0)$ in the same side of $H_Q$ as $Q$.

Thus, we can take $x=\lambda(P-Q)$ for some $0<\lambda<1$ such that $\mathcal{A}_{t^2}(f)(0)\cap(x+\mathcal{A}_{t^2}(f)(0))\neq\emptyset$ and for every $z\in \mathcal{A}_{t^2}(f)(0)\cap(x+\mathcal{A}_{t^2}(f)(0))$
\begin{itemize}
\item$t\le f(z)\le t+\varepsilon$
\item$t\le f(z-x)\le t+\varepsilon$
\end{itemize}

Notice that since $\mathcal{A}_{t^2}(f)(0)\cap(x+\mathcal{A}_{t^2}(f)(0))\neq\emptyset$, $x\in \mathcal{A}_{t^2}(f)(0)-\mathcal{A}_{t^2}(f)(0)$. On the other hand for every $z\in\mathcal{A}_{t^2}(f)(0)\cap(x+\mathcal{A}_{t^2}(f)(0))$ we have that $f(z)f(z-x)\leq(t+\varepsilon)^2\leq t-\varepsilon<t$ and for every $z\notin \mathcal{A}_{t^2}(f)(0)\cap(x+\mathcal{A}_{t^2}(f)(0))$, $f(z)f(z-x)<t$. Thus, $\mathcal{A}_t(f)(x)=\emptyset$ and $x\notin\mathcal{C}_{0,t}(f)$. Consequently, $x\in(\mathcal{A}_{t^2}(f)(0)-\mathcal{A}_{t^2}(f)(0)) \backslash \mathcal{C}_{0,t}(f)$ and $|\mathcal{C}_{0,t}(f)|<|\mathcal{A}_{t^2}(f)(0)-\mathcal{A}_{t^2}(f)(0)|$.

Thus, if there is equality in (\ref{eq:f=-g}) $\mathcal{A}_1(f)(0)$ contains a facet of every simplex $\mathcal{A}_t(f)(0)$ and by the previous lemma $f$ is the characteristic function of a simplex.
\end{proof}

We now show the characterization of the equality case in \eqref{eq:RSfunctions}.

\begin{proof}[Proof of Theorem \ref{RSfunctions} (Equality)]
Without loss of generality we assume that $||f||_{\infty}=||g||_{\infty}=1$.

Then, equality holds in \eqref{eq:RSfunctions} if and only if it holds equality on each inequality
all along the proof of \eqref{eq:RSfunctions}. In particular, we have that for almost every $t\in[0,1]$
$$
M_t|\mathcal{C}_{0,t}|={2n\choose n}\int_{\R^n}|\mathcal{A}_t(x)|dx.
$$
Notice that by Lemma \ref{ContinuityM_t} the function $M_t$ is continuous on $(0,1)$, the function $|\mathcal{C}_{0,t}|$ is continuous on $(0,1]$ since it is the volume of the level sets of the log-concave function $f\star g$, and the function $\int_{\R^n}|\mathcal{A}_t(x)|dx$ is also continuous on $(0,1]$ as a consequence of the dominated convergence theorem, since for every $t_0\in(0,1]$ and any $x\in\R^n$, $|\mathcal{A}_t(x)|$ is continuous on $t_0$ and, in a neighborhood of $t_0$, $(t_0-\varepsilon, t_0+\varepsilon)$ we have,
$$
|\mathcal{A}_t(x)|\leq|\mathcal{A}_{t_0-\varepsilon}(x)|.
$$
Thus, we have
$$
M_t|\mathcal{C}_{0,t}|={2n\choose n}\int_{\R^n}|\mathcal{A}_t(x)|dx
$$
for every $t\in(0,1)$ and, consequently, for almost every $\theta\in[0,1]$
$$
(1-\theta^{\frac{1}{n}})(-x_0(t)+\mathcal{C}_{0,t})=-x_0(t)+\mathcal{C}_{\theta,t}.
$$
As a consequence of Corollary \ref{inclusionrelation} we have equality in the last equality not only for almost every $\theta\in[0,1]$ but for every $\theta\in[0,1)$. Consequently, if $x_0(t)$ is such that $M_t=|\mathcal{A}_t(x_0)|$ we have that
\begin{enumerate}[(i)]
\item for every $\theta, t\in(0,1)$, $(1-\theta^{\frac{1}{n}})(-x_0(t)+\mathcal{C}_{0,t})=-x_0(t)+\mathcal{C}_{\theta,t}$,
\item
$$
\max_{x_0\in\R^n}\int_{\R^n}\int_0^{f\star g(x)}|\mathcal{A}_{t}(x_0)|dtdx  =
\int_{\R^n}\int_0^{f\star g(x)} \max_{x_0\in\R^n}|\mathcal{A}_{t}(x_0(t))|dtdx
$$
\item for every $x,z\in\R^n$,
$$
\min\{f\star g(x),f(z)g(x_0(t)-z)\}=f\star g(x)\,f(z)g(x_0(t)-z).
$$
\end{enumerate}

Notice that taking $\theta$ tending to 1 in (i) we have that for every $t\in(0,1)$ the maximum of $|\mathcal{A}_t(x)|$ is only attained at $x_0(t)$. Besides, by the continuity of $M_t$, the continuity of $|\mathcal{A}_t(x_0)|$ in $t$ and in $x_0$, (ii) holds if and only if $x_0(t)$ is the same for every $t\in(0,1)$, and thus we may suppose without loss of generality that $x_0(t)=0$. Thus
\begin{enumerate}[(i)]
\item for every $\theta, t\in(0,1)$, $(1-\theta^{\frac{1}{n}})\mathcal{C}_{0,t}=\mathcal{C}_{\theta,t}$,
\item for every $t\in(0,1)$, $\max_{x_0\in\R^n}|\mathcal{A}_{t}(x_0)|  = |\mathcal{A}_{t}(0)|$
\item for every $x,z\in\R^n$, $\min\{f\star g(x),f(z)g(-z)\}=f\star g(x)\,f(z)g(-z)$.
\end{enumerate}

First of all notice that if $g(x)=\chi_K(x)$ is the characteristic function of a convex body, then
$$
\mathcal{A}_t(x)=\mathcal{A}_{t^2}(f)(0)\cap(x-K)
$$
and
$$
\mathcal{C}_{\theta,t}=\mathcal{A}_{t^2}(f)(0)+_\theta K,
$$
where $\mathcal{A}_{t^2}(f)(0)+_\theta K$ denotes the $\theta$-convolution of convex bodies defined in \eqref{ConvolutionBodies}. Thus, since the $(\theta,t)$-convolution bodies of the functions are the $\theta$-convolution bodies of some convex bodies we have equality in (i) if and only if for every $t\in(0,1)$ $\mathcal{A}_{t^2}(f)(0)=-K$ is a simplex and, consequently $f(x)=g(-x)$ is the characteristic function of a simplex.

We will prove that in the equality case necessarily one of the functions is the characteristic function of a convex body.

Condition
(iii) occurs if and only if $f\star g(x)$ or $f(z)g(-z)$ equals $0$ or $1$, for every $x, z\in\R^n$. First, assume that  $f\star g(x)=1$ for every $x\in \textrm{supp }f+\textrm{supp }g$. Then for every $x\in \textrm{supp }f+\textrm{supp }g$, $\mathcal{A}_1(f)(0)\cap(x-\mathcal{A}_1(g)(0))\neq\emptyset$ and so
$$\mathcal{A}_1(f)(0)+\mathcal{A}_1(g)(0)=\textrm{supp }f+\textrm{supp }g.$$
Consequently $f$ and $g$ are characteristic functions.

Let us now assume that there exists $x\in \textrm{supp }f+\textrm{supp }g$ such that $f\star g(x)<1$. Then for every $z\in\R^n$, $f(z)g(-z)$ equals 0,1 and then, for every $t\in(0,1]$, $\mathcal{A}_t(0)=\mathcal{A}_1(0)$. In such case the function $M_t$ is constant in $(0,1]$. In particular it is also continuous on $t=1$ and then (i) also holds for $t=1$.

Notice that if $|\mathcal{A}_1(0)|=|\mathcal{A}_t(0)|=0$, then for every $t\in(0,1)$ we have that for every $\theta\in(0,1)$
$$
\mathcal{C}_{\theta,t}=\{x\in supp f+ supp g\,:\,\mathcal{A}_{t}(x)\neq\emptyset\},
$$
contradicting (i). Thus, if we have equality in \eqref{eq:RSfunctions}, $|\mathcal{A}_1(0)|>0$.

Now, if (i) holds, we fix $t\in(0,1]$ and take $x\in supp(f)+supp(g)$. If $x\notin\mathcal{C}_{0,t}$ then $\mathcal{A}_t(x)=\emptyset$. If $x\in\mathcal{C}_{0,t}$ then there exists $\theta_x\in[0,1]$ such that $x\in\partial\mathcal{C}_{\theta_x,t}$ and $x=(1-\theta_x^{\frac{1}{n}})y$ for some $y\in\partial\mathcal{C}_{0,t}$. Thus, we have equality in
\begin{align*}
\theta_x^\frac{1}{n}|\mathcal{A}_{t}(0)|^{\frac{1}{n}}
&=|\mathcal{A}_{t}(x)|^{\frac{1}{n}}
=|\mathcal{A}_{t}(\theta_x^{\frac{1}{n}}0+(1-\theta_x^{\frac{1}{n}})y)|^{\frac{1}{n}}
\\
&\geq\theta_x^{\frac{1}{n}}|\mathcal{A}_{t}(0)|^{\frac{1}{n}}+(1-\theta_x^\frac{1}{n})|\mathcal{A}_t(y)|^\frac{1}{n}
\geq\theta_x^{\frac{1}{n}}|\mathcal{A}_{t}(0)|^{\frac{1}{n}}.
\end{align*}
and, by the equality cases in Brunn-Minkowski inequality, $\mathcal{A}_t(x)$ is homothetic to $\mathcal{A}_{t}(0)$. Thus, for every
$x\in\R^n$ and $t\in(0,1]$, $\mathcal{A}_{t}(x)$ is either empty or a homothetic copy of $\mathcal{A}_{t}(0)$. Moreover, if we particularize in $t=1$,
we have that
\begin{equation*}
\begin{split}
\mathcal{A}_{1}(x)=&\{z:f(z)g(x-z)=1\}=\{z:f(z)=1\}\cap(x+\{z:g(-z)=1\})\quad\text{and}\\
\mathcal{A}_{1}(0)=&\{z:f(z)g(-z)=1\}=\{z:f(z)=1\}\cap\{z:g(-z)=1\}
\end{split}
\end{equation*}
and for any $\theta\in[0,1]$ then $\mathcal{C}_{\theta,1}$ is the $\theta$-convolution of the convex bodies
$$
\mathcal{C}_{\theta,1}=\mathcal{A}_{1}(f)(0)+_\theta\mathcal{A}_{1}(g)(0).
$$

Thus, by Proposition 2.10 in \cite{AJV}, the convex bodies $\mathcal{A}_{1}(f)(0)$ and $-\mathcal{A}_{1}(g)(0)$ are the same simplex
$\mathcal{A}_1(0)$.

Let us now assume that none of the functions $f, g$ is a characteristic function,
and we will find a contradiction.

Since $f$ and $g$ are not characteristic functions there exist $0<t_1,t_2<1$ and $z_1,z_2\in\R^n$ such that $t_1\leq f(z_1)<1$ and $t_2\leq g(-z_2)<1$.
Let us denote by $F_1$ a facet of $\mathcal{A}_1(0)$, with outer normal vector $u_1$ and contained in the hyperplane $\{x\in\R^n\,:\,\langle x,u_1\rangle=c_1\}$,
such that $\mathcal{A}_1(0)\subseteq\{x\in\R^n\,:\,\langle x,u_1\rangle\leq c_1\}$ and $\langle z_1,u_1\rangle>c_1$.
Analogously, let $F_2$ be a facet of $\mathcal{A}_1(0)$, with outer normal vector $u_2$ and contained in the hyperplane $\{x\in\R^n\,:\,\langle x,u_2\rangle=c_2\}$,
such that $\mathcal{A}_1(0)\subseteq\{x\in\R^n\,:\,\langle x,u_2\rangle\leq c_2\}$ and $\langle z_1,u_2\rangle>c_2$.

Observe that the log-concavity of $f$ and $g$ imply that
$\conv\{z_1,\mathcal{A}_1(0)\}\subset \mathcal{A}_{t_1^2}(f)(0)$ and
$\conv\{z_2,\mathcal{A}_1(0)\}\subset \mathcal{A}_{t_2^2}(\bar g)(0)$.

If $F_1=F_2$, then
$$
(\conv\{z_1,\mathcal{A}_1(0)\}\cap\conv\{z_2,\mathcal{A}_1(0)\})\backslash \mathcal{A}_1(0)\neq\emptyset.
$$
Let $z_0$ be a point in this intersection. Then $z_0\in\mathcal{A}_{t_1t_2}(0)$ since
$$
t_1t_2\leq f(z_0)g(-z_0)<1.
$$
However, this is not possible, since $\mathcal{A}_{t_1t_2}(0)=\mathcal{A}_1(0)$.

If $F_1\neq F_2$, let $x$ be a vector with a small enough norm and parallel
to the only edge contained in all $n-1$ facets $F_3,\dots,F_{n+1}$ of $\mathcal{A}_1(0)$ different from $F_1$ and $F_2$
and pointing from $F_2$ to $F_1$ such that $z_1\notin x+\mathcal{A}_1(0)$ and $\mathcal{A}_1(0)\cap(x+F_2)\neq\emptyset$,
$$
(\conv\{z_1,\mathcal{A}_1(0)\}\backslash \{z_1,\mathcal{A}_1(0)\})\cap(x+F_1)\neq\emptyset,
$$
and
$$
\mathcal{A}_1(0)\cap(x+\conv\{z_2,\mathcal{A}_1(0)\}\backslash \{z_2,\mathcal{A}_1(0)\})\neq\emptyset.
$$
Observe that for any point $z$ in the first intersection, it holds
$t_1t_2<t_1\leq f(z)g(x-z)$, whereas if $z^\prime$ is on the second, then $t_1t_2<t_2\leq f(z^\prime)g(x-z^\prime)$ and, since $x$ points from $F_2$ to $F_1$, $z^\prime$ does not belong to $x+\mathcal{A}_1(0)$.

Besides, for any other facet $F_i\cap(x+F_i)\neq\emptyset$ and for any point $z^{\prime\prime}$ in $F_i\cap(x+F_i)$ we have $f(z^{\prime\prime})g(x-z^{\prime\prime})=1$. These three types
of points belong to the set $A_{t_0t_1}(x)$, which we have proved that
is a homothetic copy of the simplex $\mathcal{A}_{t_1t_2}(0)=\mathcal{A}_1(0)$. Then, since for every facet $F_i$ there exist points in $(x+F_i)\cap A_{t_1t_2}(x)$ we have that
$x+\mathcal{A}_1(0)\subseteq A_{t_1t_2}(x)$ and since the points $z^{\prime}\in A_{t_1t_2}(x)\backslash(x+\mathcal{A}_1(0))$ the inclusion is strict. Thus
 $$|\mathcal{A}_{t_1t_2}(0)|=|\mathcal{A}_1(0)|=|x+\mathcal{A}_1(0)| < |A_{t_1t_2}(x)|,$$
contradicting (ii).

Thus, $f$ or $g$ must be a characteristic function and, consequently, $f(x)=g(-x)$ is the characteristic function of a simplex.
\end{proof}

The proof of the equality cases in (\ref{eq:RSsurfacefunctions}) follows the same lines.

\begin{proof}[Proof of Theorem \ref{thm:RSsurfacefunctions} (Equality)]
Without loss of generality we assume that $||f||_{\infty}=||g||_{\infty}=1$.

Then, equality holds in \eqref{eq:RSsurfacefunctions} if and only if it holds equality on each inequality
all along the proof of \eqref{eq:RSsurfacefunctions}. More particularly, if $x_0(t)$ is such that $M_{1,t}=W_1(\mathcal{A}_t(x_0))$ we have that
\begin{enumerate}[(i)]
\item for every $\theta, t\in(0,1)$,
$(1-\theta^{\frac{1}{n-1}})(-x_0(t)+\mathcal{C}_{0,t}^{n-1})=-x_0(t)+\mathcal{C}_{\theta,t}^{n-1}$,
\item
$$
\max_{x_0\in\R^n}\int_{\R^n}\int_0^{f\star g(x)}W_1(\mathcal{A}_{t}(x_0))dtdx
=
\int_{\R^n}\int_0^{f\star g(x)} \max_{x_0\in\R^n}W_1(\mathcal{A}_{t}(x_0(t)))dtdx
$$
\item for every $x\in\R^n, E\in\mathbb{A}_{n,1}$,
$$
\min\{f\star g(x),\max_{z\in E}f(z)g(x_0(t)-z)\}=f\star g(x)\max_{z\in E}f(z)g(x_0(t)-z)
$$
\item for every $\theta\in S^{n-1}$ and for every $z,w\in\theta^\perp$, $r,s\in\R$,
$$
\max_{s\in\R}f_z(s)g_w(r-s)=\max_{s\in\R}\min\{f_z(s)\Vert g_w\Vert,g_w(r-s)\Vert f_z\Vert_\infty\}.
$$
\end{enumerate}

Taking $\theta$ tending to 1 in (i) we have that for every $t\in(0,1]$ the maximum of $W_1(\mathcal{A}_t(x))$ is only attained at $x_0(t)$. Besides,
(ii) holds if and only if $x_0(t)$ is the same for every $t$, and thus we may suppose without loss of generality that $x_0(t)=0$. Thus
\begin{enumerate}[(i)]
\item for every $\theta, t\in(0,1)$, $(1-\theta^{\frac{1}{n-1}})\mathcal{C}_{0,t}^{n-1}=\mathcal{C}_{\theta,t}^{n-1}$,
\item for every $t\in(0,1]$, $\max_{x_0\in\R^n}W_1(\mathcal{A}_{t}(x_0))  =W_1(\mathcal{A}_{t}(0))$
\item for every $x\in\R^n, E\in\mathbb{A}_{n,1}$,
$$
\min\{f\star g(x),\max_{z\in E}f(z)g(-z)\}=f\star g(x)\max_{z\in E}f(z)g(-z),
$$
\item for every $\theta\in S^{n-1}$ and for every $z,w\in\theta^\perp$, $r,s\in\R$,
$$
\max_{s\in\R}f_z(s)g_w(r-s)=\max_{s\in\R}\min\{f_z(s)\Vert g_w\Vert,g_w(r-s)\Vert f_z\Vert_\infty\}.
$$

\end{enumerate}

As in the previous case, we have that if $g(x)=\chi_K(x)$ is the characteristic function of a convex body, then the $(n-1)$-th $(\theta,t)$-convolution bodies of the functions are the $(n-1)$-th $\theta$-convolution bodies of some convex bodies and, as it was proved in \cite{AGJ}, if $n\geq 3$ we have equality in (i) if and only if for every $t\in(0,1)$ $\mathcal{A}_{t^2}(f)(0)=-K$ is a simplex and, consequently $f(x)=g(-x)$ is the characteristic function of a simplex.

We will prove that in the equality case necessarily one of the functions is the characteristic function of a convex body.

Condition
(iii) occurs if and only if $f\star g(x)$ or $\max_{z\in E}f(z)g(-z)$ equals $0$ or $1$, for every $x\in\R^n, E\in\mathbb{A}_{n,1}$. As we have seen in the previous case, if  $f\star g(x)=1$ for every $x\in \textrm{supp }f+\textrm{supp }g$ then $f$ and $g$ are characteristic functions.

Let us now assume that there exists $x\in \textrm{supp }f+\textrm{supp }g$ such that $f\star g(x)<1$. Then for every $E\in\mathbb{A}_{n,1}$ $\max_{z\in E}f(z)g(-z)$ equals 0 or 1. Consequently, for every $t\in(0,1]$ $\mathcal{A}_t(0)=\mathcal{A}_1(0)$ because otherwise there exists some $t\in(0,1)$ and some $z\in\R^n$ such that $t\leq f(z)g(-z)<1$ and since $z\notin \mathcal{A}_1(0)$ there exist 1-dimensional affine subspaces passing through $z$ and not intersecting $\mathcal{A}_1(0)$ and for all such subspaces we would have $t\leq\max_{z\in E}f(z)g(-z)<1$. Like before, in such case (i) also holds for $t=1$.

Notice that if $W_1(\mathcal{A}_1(0))=W_1(\mathcal{A}_t(0))=0$, then for every $t\in(0,1)$ we have that for every $\theta\in(0,1)$
$$
\mathcal{C}_{\theta,t}=\{x\in supp f+ supp g\,:\,\mathcal{A}_{t}(x)\neq\emptyset\},
$$
contradicting (i). Thus, if we have equality in \eqref{eq:RSfunctions}, $W_1(\mathcal{A}_1(0))>0$.

Now, if (i) holds, we fix $t\in(0,1]$ and take $x\in supp(f)+supp(g)$. If $x\notin\mathcal{C}_{0,t}$ then $\mathcal{A}_t(x)=\emptyset$. If $x\in\mathcal{C}_{0,t}$ then there exists $\theta_x\in[0,1]$ such that $x\in\partial\mathcal{C}_{\theta_x,t}$ and $x=(1-\theta_x^{\frac{1}{n}})y$ for some $y\in\partial\mathcal{C}_{0,t}$. Thus, we have equality in
\begin{align*}
\theta_x^\frac{1}{n}W_1(\mathcal{A}_{t}(0))^{\frac{1}{n-1}}
&=W_1(\mathcal{A}_{t}(x))^{\frac{1}{n-1}}=W_1(\mathcal{A}_{t}(\theta_x^{\frac{1}{n}}0+(1-\theta_x^{\frac{1}{n}})y))^{\frac{1}{n-1}}
\\
&\geq\theta_x^{\frac{1}{n}}W_1(\mathcal{A}_{t}(0))^{\frac{1}{n-1}}+(1-\theta_x^\frac{1}{n})W_1(\mathcal{A}_t(y))^\frac{1}{n-1}
\\
&\geq\theta_x^{\frac{1}{n}}W_1(\mathcal{A}_{t}(0))^{\frac{1}{n-1}},
\end{align*}
and, by the equality cases in Brunn-Minkowski inequality for querma\ss integrtals, if $n\geq 3$ then $\mathcal{A}_t(x)$ is homothetic to $\mathcal{A}_{t}(0)$. Thus, for every
$x\in\R^n$ and $t\in(0,1]$, $\mathcal{A}_{t}(x)$ is either empty or a homothetic copy of $\mathcal{A}_{t}(0)$. Particularizing at $t=1$,
we have that for any $\theta\in[0,1]$ $\mathcal{C}_{\theta,1}^{n-1}$ is the $(n-1)$-th $\theta$-convolution of the convex bodies
$$
\mathcal{C}_{\theta,1}^{n-1}=\mathcal{A}_{1}(f)(0)+_{n-1,\theta}\mathcal{A}_{1}(g)(0).
$$

Thus, if (i) holds then by the characterization of the equality cases in \cite{AGJ} $\mathcal{A}_{1}(f)(0)$ and $-\mathcal{A}_{1}(g)(0)$ are the same simplex
$\mathcal{A}_1(0)$.

Now, if we assume that neither of the functions $f, g$ is a characteristic function, with the same proof as before we find a contradiction.
Thus, $f$ or $g$ must be a characteristic function and, consequently, $f(x)=g(-x)$ is the characteristic function of a simplex.
\end{proof}

\section{Colesanti's inequality for two functions}\label{Colesanti}

In this section we show that Colesanti's functional version of Rogers-Shephard inequality (\ref{RSColesanti}) can be extended to the case in which we consider any pair of functions  and not necessarily $g(x)=\bar f(x)$. Let us recall that similar results were obtained in \cite{AEFO}.

\begin{proof}[Proof of Theorem \ref{th_Colfg}.]
For any  $z\in\R^n$ such that $f\oplus g(z)>0$ let $x_z, y_z\in\R^n$ be such that $f\oplus g(z)=\sqrt{f(x_z)g(y_z)}$
with $2z=x_z+y_z$. Notice that $x_z$ and $y_z$ exist by Remark \ref{maximumAsplund}, since $f\oplus g(z)=\sqrt{f\star g(2z)}$. Using the log-concavity of $f$ and $g$,
\begin{equation}\label{ineqinColesanti}
f(x)g(z-x)\geq \sqrt{f(x_z)g(y_z)}
\sqrt{f(2x-x_z)g(x_z-2x)}
\end{equation}
for every $x\in\R^n$. Integrating in $x\in\R^n$
\begin{align*}
f*g(z)&\geq \sqrt{f(x_z)g(y_z)}
\int_{\R^n}\sqrt{f(2x-x_z)\bar g(2x-x_z)}\,dx
\\
&=
\frac{1}{2^n}f\oplus g(z)\int_{\R^n}\sqrt{f(x)\bar g(x)}dx.
\end{align*}
Integrating in $z\in\R^n$ we finally obtain
\[
\int_{\R^n}f(x)dx\int_{\R^n}g(x)dx\geq
\frac{1}{2^n}\int_{\R^n}f\oplus g(z)dz\int_{\R^n}\sqrt{f(x)\bar g(x)}dx,
\]
as  wanted.

Let us now characterize the equality cases. If (i) and (ii) are satisfied, then there exists $p\in\R^n$ such that
$$
\textrm{supp }f=\textrm{supp }\bar g=p+C
$$
and for every $z\in\textrm{supp }f+\textrm{supp }g$ we have that
$$
\textrm{supp }f\cap(z-\textrm{supp }g)=(p+C)\cap(z+p+C)=p+C\cap(z+C)
$$
and since $C$ has a simplicial section, this equals
$$
\textrm{supp }f\cap(z-\textrm{supp }g)=p^\prime+C
$$
for some $p^\prime\in\R^n$. We will show that $p^\prime=\frac{p}{2}+\frac{x_z}{2}$ for $x_z$ such that
$f\oplus g(z)=\sqrt{f(x_z)g(2z-x_z)}$. Notice that, as before, such $x_z$ exists, because of the continuity properties of $f$ and $g$ on their supports.
In such case we would have that for every $z\in \textrm{supp }f+\textrm{supp }g$
$$
\textrm{supp }f\cap(z-\textrm{supp }g)=\frac{x_z}{2}+\frac{1}{2}(\textrm{supp }f\cap\textrm{supp }\bar g)
$$
and then for every $z\in \textrm{supp }f+\textrm{supp }g$ and every $x\in \textrm{supp }f\cap(z-\textrm{supp }g)$
\begin{itemize}
\item $2x-x_z\in\textrm{supp }f$, and
\item $x_z-2x\in\textrm{supp }g$.
\end{itemize}
Then, inequality (\ref{ineqinColesanti}) holds with equality for the functions
$f(x)=c_1 e^{-\langle a,x\rangle}$, $x\in\textrm{supp }f$ and $g(x)=c_2e^{-\langle b,x\rangle}$, $x\in\textrm{supp }g$
 for every $x,z\in\R^n$.

In order to show that for every $z\in\textrm{supp }f+\textrm{supp }g$ we have $p^\prime=\frac{p}{2}+\frac{x_z}{2}$, notice that
$$2(\textrm{supp }f+\textrm{supp }g)= 2C-2C=C-C=\textrm{supp }f+\textrm{supp }g$$
and so, $2z\in\textrm{supp }f+\textrm{supp }g$ and for every $x\in \textrm{supp }f\cap(2z-\textrm{supp }g)$
\begin{align*}
\sqrt{f(x)g(2z-x)}
&=\sqrt{c_1c_2}e^{-\langle a,\frac{x}{2}\rangle}e^{-\langle b,z-\frac{x}{2}\rangle}
\\
&=\sqrt{c_1c_2}e^{-\langle b,z\rangle}e^{-\langle a-b,\frac{x}{2}\rangle}
\end{align*}
and so,
\begin{align*}
f\oplus g(z)
&=\sqrt{c_1c_2}e^{-\langle b,z\rangle}e^{-\min\{\langle a-b,\frac{x}{2}\rangle:{x\in\textrm{supp }f\cap(2z-\textrm{supp }g)}\}}
\\
&=\sqrt{c_1c_2}e^{-\langle b,z\rangle}e^{-\min\{\langle a-b,\bar{x}\rangle:{\bar x\in\frac{p}{2}+C\cap (z+C)}\}}.
\end{align*}
Since $(C\cap z+C)=-p+\textrm{supp }f\cap(z-\textrm{supp }g)=-p+p^\prime+C$ we have that
$$
\min_{\bar x\in\frac{p}{2}+(C\cap z+C)}\langle a-b,\bar x\rangle
=
\min_{\bar x\in-\frac{p}{2}+p^\prime+C}\langle a-b,\bar x\rangle
$$
and since $\langle a-b, x\rangle\geq 0$ for every $x\in C$, the minimum is attained when $\bar x=\frac{x_z}{2}=-\frac{p}{2}+p^\prime$. Thus $p^\prime=\frac{p}{2}+\frac{x_z}{2}$.

Let us now prove that
(i) and (ii) are necessary conditions for equality in (\ref{eq:Colfg}) to hold.
We can assume, without loss of generality, that $f(x_0)=\Vert f\Vert_\infty=\Vert g\Vert_\infty=g(y_0)=1$ and let us write $f(x)=e^{-u(x)}$ and $g(x)=e^{-v(x)}$ for some convex functions $u,v$. Notice that, since $\Vert f\Vert_\infty=\Vert g\Vert_\infty=1$, $u$ and $v$ take values in $[0,+\infty]$.

Equality in (\ref{eq:Colfg}) happens if and only if for every $z\in \textrm{supp }f+\textrm{supp }g$ and every $x\in \textrm{supp }f\cap(z-\textrm{supp }g)$ we have equality in (\ref{ineqinColesanti}). Thus, for every $z\in \textrm{supp }f+\textrm{supp }g$ the support of both functions as functions of $x\in\R^n$ must be the same and so
$$
\textrm{supp }f\cap(z-\textrm{supp }g)
=
\frac{x_z}{2}+\frac{1}{2}(\textrm{supp }f\cap(-\textrm{supp }g)),
$$
where $x_z$ is such that $f\oplus g(z)=\sqrt{f(x_z)g(2z-x_z)}$.
Notice that in particular this implies that $\textrm{supp }f\cap(-\textrm{supp }g)$ is full-dimensional, since we are assuming that $\textrm{supp }f$ and $\textrm{supp }g$ are full-dimensional and then there exists some $z\in\textrm{supp }f+\textrm{supp }g$ such that $\textrm{supp }f\cap(z-\textrm{supp }g)$ is full-dimensional and then $(\textrm{supp }f\cap(-\textrm{supp }g))$ is full-dimensional.

Besides, for every $z\in \textrm{supp }f+\textrm{supp }g$ we have equality in (\ref{ineqinColesanti}) for every $x\in \textrm{supp }f\cap(z-\textrm{supp }g)$ if and only if
\begin{align*}
u\left(\frac{1}{2}(2x-x_z)+\frac{x_z}{2}\right)&=\frac{1}{2}(u(2x-x_z)+u(x_z))
\\
v\left(\frac{1}{2}(x_z-2x)+\frac{y_z}{2}\right)&=\frac{1}{2}(v(x_z-2x)+v(y_z))
\end{align*}
for every $x\in \textrm{supp }f\cap(z-\textrm{supp }g)$, where $x_z,y_z\in\R^n$ are such that $f\oplus g(z)=\sqrt{f(x_z)g(y_z)}$. In particular, for $z_0=\frac{x_0+y_0}{2}$ we have that $f\oplus g(z_0)=\sqrt{f(x_0)g(y_0)}$ and so
\begin{align*}
u\left(\frac{1}{2}(2x- x_0)+\frac{ x_0}{2}\right)&=\frac{1}{2}u(2x- x_0)
\\
v\left(\frac{1}{2}( x_0-2x)+\frac{ y_0}{2}\right)&=\frac{1}{2}v( x_0-2x).
\end{align*}
Consequently, for every $x^\prime=x- x_0$ and $y^\prime=\frac{ x_0}{2}-\frac{ y_0}{2}-x$
\begin{align*}
u\left( x_0+x^\prime\right)&=\frac{1}{2}u( x_0+2x^\prime)
\\
v\left( y_0+y^\prime\right)&=\frac{1}{2}v( y_0+2y^\prime).
\end{align*}
Thus, $\textrm{supp }f\cap(z_0-\textrm{supp }g)=\frac{x_0}{2}+\frac{1}{2}(\textrm{supp }f\cap(-\textrm{supp }g))$ is a closed convex cone $x_0+C$ (with $C$ a cone with vertex at 0) and so
$$
\textrm{supp }f\cap(-\textrm{supp }g)=
x_0+C
$$
and for every $z\in\textrm{supp }f+\textrm{supp }g$
$$
\textrm{supp }f\cap(z-\textrm{supp }g)
=
\frac{x_z}{2}+\frac{x_0}{2}+C.
$$

Furthermore, $\textrm{supp }g\cap( z_0-\textrm{supp }f)= z_0-\textrm{supp }f\cap( z_0-\textrm{supp }g)= z_0- x_0-C $ is a closed convex cone $y_0+C^\prime$. This implies that $C^\prime=-C$.
Besides, if the vertex of the cone is unique, $z_0- x_0= y_0$ and then $ y_0=- x_0$ which implies  $ z_0=0$. If the vertex is not unique then, since both $ z_0- x_0$ and $ y_0$ are  vertices of the cone $ y_0-C$, also $ y_0+2( z_0- x_0- y_0)= y_0-2 z_0=- x_0$ and so $y_0-C=-x_0-C$ and $-x_0+ C= y_0+C$.

On the other hand, for any $z\in\textrm{supp }f+\textrm{supp }g$ if $x\in\textrm{supp }f\cap(z-\textrm{supp }g)$, then
$$
2x-x_z\in  x_0+C.
$$
If equality holds in (\ref{ineqinColesanti}) then $u$ is affine in any segment that connects $x_0+C$ with a point $x_z$ and $v$ is affine in any segment that connects $- x_0-C$ and some $y_z$.

Let us now take $z\in  x_0+\textrm{supp }g$. Then $- x_0+z\in\textrm{supp }g$ and, since $ x_0\in\textrm{supp }f$ we have that
$- x_0+z\in(z-\textrm{supp }f)\cap\textrm{supp }g=z-\textrm{supp }f\cap(z-\textrm{supp }g)=z-\frac{x_z}{2}-\frac{ x_0}{2}-C$ and so $\frac{x_z}{2}\in\frac{ x_0}{2}-C$ and
$$
x_z\in x_0-C.
$$
Consequently $x_z= x_0$. Otherwise, consider the ray from $x_z$ that passes through $ x_0$. Since $x_z\in  x_0-C$ any point $p$ in this ray such that the segment $[x_z,p]$ contains $ x_0$ is contained in $ x_0+C$ and since $u\geq 0$ is affine in the segment $[x_z,p]$ and $u( x_0)=0$ then $u=0$ for any such $p$, contradicting the integrability of $f$.

Thus, for any $z\in  x_0+\textrm{supp }g$ we have that $x_z= x_0$ and $y_z=2z- x_0$.
Notice that then for every $z\in x_0+\textrm{supp }g$ we have that
$$
\textrm{supp }f\cap(z-\textrm{supp }g)
=
\frac{x_z}{2}+\frac{ x_0}{2}+C
= x_0+C
$$
as $x_z= x_0$. Consequently, $\textrm{supp }f=\textrm{supp }f\cap(-\textrm{supp }g)$ since for every
$y\in \textrm{supp }f$, as $C$ is a full-dimensional cone, we can take $w\in C$ such that
$y\in- w+( x_0+C)$. Thus, if we take $z=- w\in-C= x_0-( x_0+C)\subset  x_0+\textrm{supp }g$,
we have that $y\in z+( x_0+C)\subset z-\textrm{supp }g$ and so $y\in\textrm{supp }f\cap(z-\textrm{supp }g)=\textrm{supp }f\cap(-\textrm{supp }g)$. Consequently $\textrm{supp }f= x_0+ C\subseteq-\textrm{supp }g$.

Analogously,
 take $z\in - x_0+\textrm{supp } f$. Since $- x_0\in\textrm{supp }g$, we have that $ x_0+z\in\textrm{supp }f\cap(z-\textrm{supp }g)=\frac{x_z}{2}+\frac{ x_0}{2}+C$ and, consequently $\frac{y_z}{2}\in-\frac{ x_0}{2}+C$.
 Thus $y_z\in - x_0+C= y_0+C$. Consequently $y_z= y_0$.
 Otherwise, consider the ray from $y_z$ that passes through $y_0$. Since $y_z\in  y_0+C$ any point $p$ in this ray such that the segment $[y_z,p]$ contains $ y_0$ is contained in $ y_0-C$ and since $v\geq 0$ is affine in the segment $[y_z,p]$ and $v( x_0)=0$ then $v=0$ for any such $p$, contradicting the integrability of $g$.

Thus, for any $z\in - x_0+\textrm{supp }f$ we have that $x_z=2z- y_0$ and $y_z= y_0$. Notice that then for every $z\in- x_0+\textrm{supp }f$ we have that
$$
\textrm{supp }f\cap(z-\textrm{supp }g)=
\frac{x_z}{2}+\frac{ x_0}{2}+C=z-\frac{ y_0}{2}+\frac{ x_0}{2}+C=z+ x_0+C,
$$
since $ x_0+ C=- y_0+C$. Consequently, $-\textrm{supp }g=\textrm{supp }f\cap(-\textrm{supp }g)$ since for every
$y\in -\textrm{supp }g$, as $C$ is a full-dimensional cone, we can take $w\in C$ such that
$y\in- w+( x_0+C)$. Thus, if we take $z= w\in C=- x_0+\textrm{supp }g$
we have that $y+z\in\textrm{supp }f\cap(z-\textrm{supp }g)=z+ x_0+C$ and so $y\in x_0+C=\textrm{supp }f$.

Consequently $-\textrm{supp }g\subseteq\textrm{supp }f$ and so $-\textrm{supp }g=\textrm{supp }f= x_0+C$.

Now let us see that $v$ is affine on $- x_0-C$. Let us take $x,y\in- x_0-C$ and consider $z=\frac{y}{2}+\frac{ x_0}{2}\in-C= x_0+\textrm{supp }g$.
Then, for this $z$, $y_z=y$ and since $x\in- x_0-C$, $v$ is affine in the segment that connects $x$ and $y$. Consequently, $v$ is affine on
$- x_0-C$ and so $g(x)=c_2e^{-\langle b,x\rangle}$ on $- x_0-C$. Thus, $C$ does not contain any straight line $l$ since otherwise, as $v$ is affine and positive, it must be constant on the line $l$ and so $C$ has only one vertex and $ x_0=- y_0$ and $\langle b,x\rangle<0$ for every $x\in C\backslash \{0\}$.

Analogously, $u$ is also affine on $ x_0+C$ since for any  $x,y\in x_0+C$, if we take $z=\frac{x}{2}-\frac{ x_0}{2}\in C=- x_0+\textrm{supp }f$ we have that for this $z$, $x_z=x- x_0- y_0=x$ and so, $u$ is affine on $ x_0+C$ and $f(x)=c_1e^{-\langle a, x\rangle}$ on $ x_0+C$.
Since $\Vert f\Vert_\infty=f( x_0)$ we have that $\langle a,x\rangle>0$ for every $x\in C$.


Finally, considering the section of $C$ by a hyperplane, since the intersection of this section with any of its translates is homothetic to itself,
 the section must be a simplex.
\end{proof}

\begin{rmk}
Theorem \ref{th_Colfg} becomes inequality \eqref{RSColesanti}
if $g(x)=\bar f(x)$ because $f\oplus g(z)=\Delta f(z)$. Moreover,
it also recovers (\ref{RSuniontwobodies}) when
we particularize
$f(x)=e^{-h_K(x)}, g(x)=e^{-h_L(-x)}$, where $h_K$ and $h_L$ are the
support functions of two convex bodies $K$ and $L$ that contain the origin.
Then
$f\oplus g(x)=e^{-h_{K\cap L}(x)}$ and $\sqrt{f(x)g(x)}=e^{-h_{\frac{K-L}{2}}(x)}$,
and since
\[
\int_{\R^n}e^{-h_K(x)}dx=n!|K^\circ|,
\]
then \eqref{eq:Colfg} becomes
$$
|(K\cap L)^\circ|\left|\left(\frac{K-L}{2}\right)^\circ\right|\leq
2^n|K^\circ||L^\circ|.
$$
From the characterization of the equality cases in (\ref{eq:Colfg}), this only converges to equality when we consider sequences of sets $(K_n^\circ)_n$ and $(-L_n^\circ)_n$ converging to simplices with 0 in one of the vertices and the same outer normal vectors at the facets that pass through the origin.

Taking into account that $(K\cap L)^\circ=\conv\{K^\circ,L^\circ\}$ and $\frac{K-L}{2}\subset
\conv\{K,-L\}$, if we change the role of $K^\circ$ and $L^\circ$ by $K$ and
$-L$ for simplicity, then
\[
|K\cap L||\conv\{K,-L\}|\leq
\left|\left(\frac{K^\circ-L^\circ}{2}\right)^\circ\right||\conv\{K,-L\}|\leq
2^n|K||L|,
\]
showing the assertion and slightly strengthening \eqref{RSuniontwobodies}. In order to have equality in \eqref{RSuniontwobodies} we must have $K$ and $L$ be simplices with 0 in one of the vertices and the same outer normal vectors at the facets that pass through the origin and $\frac{K+L}{2}=\conv\{K,L\}$. Thus $K=L$ is a simplex. Otherwise there exists a direction $\theta\in S^{n-1}$ such that $h_K(\theta)< h_L(\theta)$ (or $h_L(\theta)<h_K(\theta)$). Thus, $h_{\frac{K+L}{2}}(\theta)<h_L(\theta)\leq h_{\conv\{K,L\}}(\theta)$ (or $h_{\frac{K+L}{2}}(\theta)<h_K(\theta)\leq h_{\conv\{K,L\}}(\theta)$).
\end{rmk}

\centerline{ACKNOWLEDGEMENTS}

Part of this work was carried out at the `Instituto de Matem\'aticas de la Universidad de Sevilla' (IMUS)
where the D. Alonso was invited and B. Gonz\'alez fulfilled the program
`Ayudas para estancias cortas postdoctorales en el IMUS', and they are thankful for the invitation and for the good working conditions and environment there.

D. Alonso is partially supported by `Institut Universitari de
Matem\`atiques i Aplicacions de Castell\'o', Spanish Ministry of Sciences and Innovation (MICINN) project
MTM2013-42105-P and BANCAJA project P1-1B2014-35. 
B. Gonz\'alez is
partially supported by Spanish Ministry of Economy and Competitiveness (MINECO) project MTM2012-34037.

C. H. Jim\'enez and R. Villa are supported by MINECO project MTM2012-30748 and C. H. Jim\'enez is also supported by Capes and IMPA.


\begin{thebibliography}{99?}
\bibitem{AGJ}{\sc Alonso-Guti\'errez D., Gonz\'alez Merino B., Jim\'enez C.H.}
\textit{Volume inequalities for the $i$-th convolution bodies.} J. Math. Anal. Appl. {\bf 424 (1)} (2015), pp. 385-401.

\bibitem{AJV}{\sc Alonso-Guti\'errez D., Jim\'enez C.H., Villa R.}
\textit{Brunn-Minkowski and Zhang inequalities for convolution bodies.}
Adv. in Math. {\bf 238} (2013), pp. 50--69.

\bibitem{AEFO}{\sc Arstein S., Einhorn K., Florentin D.I., Ostrover Y.}
\textit{On Godbersen's conjecture.} Geom. Dedicata {\bf 178 (1)} (2015), pp. 337-–350.

\bibitem{AGM}{\sc Arstein S., Giannopoulos A., Milman V.}
\textit{Asymptotic Geometric Analysis, Part I}
Mathematical Surveys and monographs, {\bf 22}, (2015), American Mathematical Society, Providence, Rhode Island,

\bibitem{AKM}{\sc Artstein S., Klartag M., Milman V.}
\textit{The Santal\'o point of a function and a functional form of Santal\'o inequality.}
Mathematika {\bf 51} (2004), pp. 33--48.

\bibitem{AKSW}{\sc Artstein S., Klartag M., Sch\"utt C., Werner E.}
\textit{Functional affine-isoperimetry and an inverse logarithmic Sobolev inequality}
J. Funct. Anal., {\bf 262} (2012), no. 9, pp. 4181--4204.

\bibitem{ASY}
{\sc Aubrun G., Szarek S. J., Ye D.} \textit{Entanglement thresholds for random induced states.}
 Comm. Pure Appl. Math.  {\bf 67}  (2014),  no. 1, 129--171.

\bibitem{BBCG}{\sc
Bakry D., Barthe F., Cattiaux P., Guillin A.} \textit{A simple proof of the Poincar\'e inequality for a large class of probability measures including the log-concave case,} Electron. Commun. Probab., {\bf13} (2008), pp. 60–-66.

\bibitem{BE}{\sc
Bakry D., Émery M.} \textit{Diffusions hypercontractives,} in Proc. S\'eminaire de probabilit\'es, XIX, 1983/84, Berlin, Germany, 1985, {\bf1123}, Lecture Notes in Math., pp. 177–-206.

\bibitem{Ball}
{\sc Ball K.} \textit{Logarithmically concave functions and sections of convex sets in $R^{n}$,} Studia Math., {\bf 88} (1988), no. 1, pp. 69--84.

\bibitem{Be}{\sc Bezdek K.} \textit{Illuminating spindle convex bodies and minimizing the volume of
 spherical sets of constant width.} Discrete Comput. Geom.  {\bf 47} (2012),  no. 2,  pp. 275--287.

\bibitem{B}
{\sc Bobkov S. G.} \textit{Isoperimetric and analytic inequalities for log-concave probability measures,} Ann. Probab., {\bf 27} (1999), no. 4, pp. 1903–-1921.

\bibitem{BCF}{\sc Bobkov S. G., Colesanti A., Fragal\'a I.}
\textit{Quermassintegrals of quasi-concave functions
and generalized Pr\'ekopa-Leindler inequalities.} Manuscripta Mathematica, {\bf 143}, no. 1, (2014) pp 131--169

\bibitem{BM1}
{\sc Bobkov S. G., Madiman M.} \textit{The entropy per coordinate of a random vector is highly constrained under convexity conditions.} IEEE Transactions on Information Theory, {\bf 57} (2011), no. 8, pp. 4940--4954.

\bibitem{BM2}
{\sc Bobkov S. G., Madiman M.} \textit{On the problem of reversibility of the entropy power inequality.} Limit Theorems in Probability, Statistics and Number Theory, Festschrift in honor of F. G\"otze's 60th birthday, P. Eichelsbacher et al. (ed.), Springer Proceedings in Mathematics and Statistics {\bf 4}2, pp. 61--74, Springer-Verlag, (2013).
\bibitem{Bo}{\sc Borell C.}
\textit{Convex set functions in $d$-space.}
Period. Math. Hungar. {\bf 6} (1975), no. 2, pp. 111–-136.

\bibitem{C}{\sc Colesanti A.}
\textit{Functional inequalities related to the Rogers-Shephard inequality.}
Mathematika {\bf 53} (2006) pp. 81--101.

\bibitem{CF}{\sc Colesanti A., Fragal\'a I.}
\textit{The area measure of log-concave functions and related inequalities.} Adv. Math. {\bf 244}, (2013), pp. 708--749.

\bibitem {FM}{\sc Fradelizi M., Meyer M.} \textit{Some functional forms of Blaschke-Santal ́o inequality.} Math. Z. {\bf 256} (2007), no. 2, pp. 379–395.

\bibitem{G} {\sc Gardner R. J.} \textit{Geometric tomography, Second edition}, Encyclopedia of Mathematics and its Applications, {\bf 58} (2006), Cambridge University Press, Cambridge.

\bibitem{Kl} {\sc Klartag B.} \textit{On convex perturbations with a bounded isotropic constant.} Geom. Funct. Anal. {\bf 16} (2006), no. 6, pp. 1274-1290.

\bibitem{KM}{\sc Klartag B., Milman V. D.} \textit{Geometry of log-concave
functions and measures.} Geom. Dedicata {\bf 112} (2005), no. 3, pp. 169--182.

\bibitem{Ku} {\sc Kuperberg G.} \textit{From the Mahler conjecture to Gauss linking integrals.}
 Geom. Funct. Anal.  {\bf 18}  (2008),  no. 3, pp. 870--892.

\bibitem{M}
{\sc Milman V. D.} \textit{Geometrization of Probability, Geometry and dynamics of groups and spaces}, 647–667, Progr. Math., {\bf 265} (2008), Birkhuser, Basel.


\bibitem{MP} {\sc Milman V. D.,  Pajor A.} \textit{Entropy and asymptotic geometry of non-symmetric convex bodies.}
 Adv. Math.  {\bf 152}  (2000),  no. 2, pp. 314--335.

\bibitem{Pa} {\sc Paouris G.} \textit{Concentration of mass on convex bodies.} Geom. Funct. Anal. {\bf 16} (2006), no. 5, pp. 1021-1049.

\bibitem{P}{\sc Pr\'ekopa A.} \textit{Logarithmic concave measures and functions.} Acta Scientiarum Mathematicarum {\bf 34} (1973), no. 1, pp. 334--343.

\bibitem{RS}{\sc Rogers C. A., Shephard G. C.}
\textit{The difference body of a convex body.}  Arch. Math. {\bf 8} (1957), pp. 220--233

\bibitem{RS2}{\sc Rogers C. A., Shephard G. C.}
\textit{Convex bodies associated with a given convex body.} J. Lond. Math. Soc. {\bf 33} (1958), pp. 270--281.

\bibitem{Ro}{\sc Rotem L.}
\textit{On the mean width of log-concave functions.} Geometric aspects of functional analysis.
Springer Berlin Heidelberg, 2012, pp. 355--372.

\bibitem{R}{\sc Rudelson M.} \textit{Distances between non-symmetric convex bodies and the MM$^*$-estimate.}
 Positivity {\bf 4} (2000), no. 2, pp. 161--178.

\bibitem{S}{\sc Schneider R.}
\textit{Convex bodies: The Brunn-Minkowski Theory.}
Cambridge University Press, Cambridge, (1993).

\bibitem{SWZ} {\sc Szarek S. J., Werner E., Życzkowski K.} \textit{Geometry of sets of quantum maps: a generic positive map acting on a
 high-dimensional system is not completely positive.}
 J. Math. Phys.  {\bf 49}  (2008),  no. 3, 032113, 21 pp.

\bibitem{Z}{\sc Zhang G.}  \textit{The affine Sobolev inequality.} J. Differential Geom, {\bf 53} (1999), no. 1, pp. 183--202.





\end{thebibliography}
\end{document}